\documentclass[12pt]{amsart}
\usepackage{amscd, amsfonts, amssymb}
\usepackage{fullpage}
\usepackage{latexsym}
\usepackage{verbatim}
\usepackage{enumerate}

\newcommand\inverse{{^{-1}}}

\newcommand\ra{\rightarrow}

\newcommand{\iso}{\cong}
\newcommand{\NN}{{\mathbb N}}
\newcommand{\ZZ}{{\mathbb Z}}
\newcommand{\QQ}{{\mathbb Q}}
\newcommand{\RR}{{\mathbb R}}

\DeclareMathOperator{\Aut}{Aut}

\DeclareMathOperator{\diag}{diag}

\DeclareMathOperator{\GL}{GL}
\DeclareMathOperator{\Gal}{Gal}
\DeclareMathOperator{\SL}{SL}

\DeclareMathOperator{\sgn}{sgn}
\DeclareMathOperator{\Mat}{Mat}

\numberwithin{equation}{section}

\newtheorem{thm}[equation]{Theorem}

\newtheorem{lem}[equation]{Lemma}
\newtheorem{cor}[equation]{Corollary}
\newtheorem{prop}[equation]{Proposition}
\newtheorem{conj}[equation]{Conjecture}

\theoremstyle{definition}
\newtheorem{defn}[equation]{Definition}
\newtheorem{exmp}[equation]{Example}
\theoremstyle{remark}
\newtheorem{rem}[equation]{Remark}
\theoremstyle{remark}
\newtheorem{rems}[equation]{Remarks}

\newcommand{\ovl}{\overline}


\title[The strong Centre Conjecture]
{The strong Centre Conjecture:\\ an invariant theory approach}

\author[M.\  Bate]{Michael Bate}
\address
{Department of Mathematics,
University of York,
York YO10 5DD,
United Kingdom}
\email{michael.bate@york.ac.uk}

\author[B.\ Martin]{Benjamin Martin}
\address
{Department of Mathematics,
University of Auckland,
Private Bag 92019,
Auckland 1142,
New Zealand}
\email{ben.martin@auckland.ac.nz}

\author[G. R\"ohrle]{Gerhard R\"ohrle}
\address
{Fakult\"at f\"ur Mathematik,
Ruhr-Universit\"at Bochum,
D-44780 Bochum, Germany}
\email{gerhard.roehrle@rub.de}

\begin{document}

\begin{abstract}
The aim of this paper is to 
describe an approach to
a strengthened form of J.\ Tits' Centre Conjecture for spherical buildings.
This is accomplished by generalizing a fundamental result of G.R.\ Kempf
from Geometric Invariant Theory and interpreting this generalization
in the context of  spherical buildings.
We are able to recapture the conjecture
entirely in terms of our generalization of Kempf's notion of a state.
We demonstrate the utility of this approach by
proving the Centre Conjecture in some special cases.
\end{abstract}

\maketitle

\setcounter{tocdepth}{1}
\tableofcontents

\section{Introduction}
\label{sec:intro}
The main focus of this paper is the long-standing Centre Conjecture of J.\ Tits about the structure of
convex subsets of spherical buildings.
Roughly speaking, the Centre Conjecture asserts that a convex subset $\Sigma$ of a spherical building $\Delta$ should be
a subbuilding in an appropriate sense, or should contain a natural centre --
a point of $\Sigma$ which is fixed by all automorphisms of $\Delta$ that stabilize $\Sigma$.
See Conjecture \ref{conj:tcc} below for a precise statement and references.
Apart from its independent interest, this conjecture arises in many areas of mathematics,
particularly the theory of reductive linear algebraic groups and their subgroups (\cite{serre2}, \cite{BMR}, \cite{BMR:tits})
and the study of algebraic groups acting on algebraic varieties, which we 
refer to as Geometric Invariant Theory (GIT) (\cite{mumford}, \cite{rousseau}, \cite{GIT}).

In this paper we give an approach to the Centre Conjecture using 
GIT. A key idea is to consider not just 
convex subcomplexes  of $\Delta$ but instead a more general class
of convex \emph{subsets}. We study this strengthened version of 
the Centre Conjecture.

The original formulation of the Centre Conjecture in the 1950s came about
as a possible way to answer a fundamental question
about the subgroup structure of a reductive algebraic group $G$, \cite[Lem.\ 1.2]{tits0},
later answered by Borel and Tits via different means \cite{boreltits2}.
The Centre Conjecture also occurs naturally in GIT, when one is considering
the notion of unstable points in an affine $G$-variety \cite[Ch.\ 2]{mumford}.
In this context, solutions to the Centre Conjecture were found by Kempf \cite{kempf} and Rousseau \cite{rousseau} in the 1970s;
see Remark \ref{rem:getkempf}.
There has also been a recent renewal of interest in the Centre Conjecture from building theorists, 
culminating in a proof of the Centre Conjecture for convex \emph{subcomplexes} of thick spherical buildings.
This proof relies on case-by-case studies by M\"uhlherr and Tits \cite{muhlherrtits},
Leeb and Ramos-Cuevas \cite{lrc}, and Ramos-Cuevas \cite{rc}.

The purpose of this paper is to bring together some of the GIT-methods of Kempf \cite{kempf}, Rousseau \cite{rousseau}
and Hesselink \cite{He} in the context of the Centre Conjecture.
We concentrate in particular on the work of Kempf \cite{kempf}, who never makes explicit
the connection between his work and the Centre Conjecture.
By carefully modifying some of Kempf's key constructions, we are able to significantly extend his results.
In the original context of GIT, this gives new results about instability
for $G$-actions on affine varieties (see Remark \ref{rem:twosettings}).
In the context of spherical buildings and the Centre Conjecture,
our extensions provide a
scheme for attacking the Centre Conjecture for a large class of convex subsets of $\Delta_G$, the spherical building of $G$.
By combining these two points of view,
we are able to apply our methods to provide uniform (rather than case-by-case) proofs of some cases of the Centre Conjecture.
Our methods have the additional advantage of being \emph{constructive} -- not only do we prove the
existence of a centre, but we give a way of finding this centre --
and they also cover new cases of the Centre Conjecture (in general, the subsets coming from GIT are not subcomplexes of $\Delta_G$).
On the other hand, we restrict attention in this paper to finding ``$G$-centres'' for
convex subsets
of $\Delta_G$ -- that is, we restrict attention to those building automorphisms
that come from $G$.
Our main reason for this is to keep the exposition more accessible;
in the final section we briefly indicate how our methods may be extended to cover automorphisms which do not come from $G$.

Our principal result, Theorem \ref{thm:partialTCC}, gives a necessary and
sufficient criterion for the existence of a $G$-centre of a
convex subset of $\Delta_G$
in terms of our generalization of Kempf's notion of a state.
While GIT-methods have been used in the past to approach
the Centre Conjecture, cf.\
\cite{rousseau}, Theorem \ref{thm:partialTCC}
provides a new strategy for an attack on this conjecture.

The paper is laid out as follows.
In Section \ref{sec:prelims} we collect 
a wide range of prerequisites.
Starting with basic properties of algebraic groups and their sets of cocharacters and characters,
we construct the vector and spherical buildings associated to a semisimple group $G$.
This allows us to give a formal statement of Tits' Centre Conjecture \ref{conj:tcc}.
We also provide some basic material on convex cones, Serre's notion of $G$-complete reducibility,
and the notion of instability in GIT.
In Section \ref{sec:quasistates}, we proceed with our generalization of Kempf's work from \cite{kempf}.
This section lies at the heart of the paper, and culminates with our Theorem \ref{thm:kempf2.2},
which generalizes Kempf's key theorem \cite[Thm.\ 2.2]{kempf}.
When translated into the language of spherical buildings in Section \ref{sec:buildings}, our results give an
equivalent formulation of the Centre Conjecture
in terms of our generalization of Kempf's notion of a state: see Theorem \ref{thm:bldgversion},
Theorem \ref{thm:partialTCC}, and Remark \ref{rems:howtoprovetcc}(i).
In particular, Theorem \ref{thm:partialTCC}
gives a complete characterization of the existence of a $G$-centre
of a convex subset of $\Delta_G$.

In Section \ref{sec:gitstcc}, we apply our strengthening of Kempf's
results to situations arising from GIT.
In particular, we show how to recover existing results in the literature (especially from \cite{kempf}, \cite{GIT})
from our constructions; see Remark \ref{rem:getkempf}.
Subsequently, we then apply our methods
to prove the Centre Conjecture in some special cases; see Theorem \ref{thm:precise}, Theorem \ref{thm:tcc},
and Theorem \ref{thm:apt}.
It is rather striking that
these last two results provide applications of our new GIT-methods
to situations which have no apparent connection with GIT.
This supports our view that these
methods provide valuable insight into the Centre Conjecture.

The final section of the
paper briefly discusses various ways in which our results can be extended,
depending on the situation at hand.

\section{Preliminaries}
\label{sec:prelims}


\subsection{Basic notation}
\label{subsec:notn}
Throughout the paper (except in part of Section \ref{sec:extra}), $G$ denotes a 
semisimple linear algebraic group defined over an algebraically closed field $k$.  Many of our results hold for an arbitrary reductive algebraic group $G$ (see Section \ref{sub:redgps}).
By a subgroup of $G$ we mean a closed subgroup.  Let $H$ be a subgroup of $G$.
We denote by $R_u(H)$ the unipotent radical of $H$.

Let $T$ be a maximal torus of $G$ and let
$\Psi(G,T)$ denote the set of roots of $G$ with respect to $T$.
For $\alpha \in \Psi(G,T)$, we denote the corresponding root subgroup
of $G$ by $U_\alpha$. For a $T$-stable subgroup $H$ of $G$, we denote the
set of roots of $H$ with respect to $T$ by
$\Psi(H,T) := \{\alpha \in \Psi(G,T) \mid U_\alpha \subseteq H\}$.

Whenever a group $\Gamma$ acts on a set $\Omega$, we let $C_\Gamma(\omega)$ denote the stabilizer in $\Gamma$ of $\omega\in \Omega$.
If $\Sigma$ is a subset of $\Omega$, we let $N_\Gamma(\Sigma)$ denote the subgroup of elements of $\Gamma$ that stabilize $\Sigma$ setwise.

\subsection{Cocharacters and parabolic subgroups}
\label{sub:cochar}

For any linear algebraic group $H$, we let $Y_H$, $X_H$ denote the
sets of cocharacters and characters of $H$, respectively.
When $H=G$, we drop the suffix and write $Y=Y_G$.
We write $X$ for the \emph{disjoint} union of the $X_T$, where $T$ runs over the maximal tori of $G$.
If $H$ is a torus, then $Y_H$ and $X_H$ are abelian groups which we write additively:
e.g., if $\lambda, \mu \in Y_H$ and $a \in k^*$, then
$(\lambda+\mu)(a) := \lambda(a)\mu(a)$.
For any torus $H$, we denote by $\langle \ {,}\  \rangle$ the usual
pairing $Y_H\times X_H\ra \ZZ$.
We have a left action of $G$ on $Y$ given by $(g,\lambda)\mapsto g\cdot\lambda$,
where $(g\cdot \lambda)(a):=g\lambda(a)g^{-1}$ for $a \in k^*$.
Moreover, there is a left action of $G$ on $X$ given by $(g,\beta)\mapsto g_!\beta$,
where $(g_!\beta)(x)= \beta(g^{-1}xg)$ for $x \in G$.
Note that if $H$ is a subgroup of $G$, $\lambda \in Y_H$, $\beta \in X_H$ and
$g\in G$, then
$g\cdot \lambda \in Y_{gHg^{-1}}$ and
$g_!\beta \in X_{gHg^{-1}}$.
If $H$ is a torus of $G$, $\lambda\in Y_H$, $\beta\in X_H$, and $g\in G$, we have
\begin{equation}
\label{eqn:ipinvce}
 \langle g\cdot\lambda,g_!\beta \rangle= \langle \lambda,\beta\rangle.
\end{equation}

We recall \cite[Def.\ 4.1]{GIT}.

\begin{defn}
\label{def:norm}
A \emph{norm} on $Y$ is a
non-negative real-valued function $\left\|\,\right\|$ on $Y$
such that
\begin{itemize}
\item[(a)] $\left\| g \cdot \lambda \,\right\| = \left\| \lambda \,\right\|$
for any $g \in G$ and any $\lambda \in Y$;
\item[(b)] for any maximal torus $T$ of $G$, there is a positive definite
integer-valued form $(\ {,}\ )$ on $Y_T$ such that
$(\lambda, \lambda) = \left\| \lambda \,\right\|^2$ for any $\lambda \in Y_T$.
\end{itemize}
\end{defn}

Such norms always exist, as follows from \cite[Lem.\ 2.1]{kempf}.
From now on, we fix a norm $\left\|\,\right\|$ on $Y$.

We now extend the notion of a cocharacter.
For the rest of the paper, whenever it is not specified,
the letter $K$ stands for either one of $\QQ$ or $\RR$.
Let $H$ be a subgroup of $G$.
Define $Y_H(\QQ)$ to be the quotient of $Y_H\times \NN_0$ by the equivalence relation:
$(\lambda,m)\equiv (\mu,n)$ if $n\lambda=m\mu$.
In particular, $Y_H(\QQ)\iso Y_H\otimes_\ZZ \QQ$ if $H$ is a torus.
For any maximal torus $T$ of $G$, we define $Y_T(\RR)= Y_T(\QQ)\otimes_\QQ \RR$.
Given $\lambda,\mu\in Y_T(K)$, we denote by $[\lambda,\mu]$ the line segment
$\{a\lambda+ b\mu\mid a,b \in K, a,b\geq 0, a+b=1\}$ between $\lambda$ and $\mu$ in $Y_T(K)$.
It is clear from the definition that this line segment does not depend on
the choice of $T$ with $\lambda,\mu \in Y_T(K)$.

Let $T$ be a maximal torus of $G$, and let $\Psi(G,T)$
be the set of roots of $G$ with respect to $T$.
If $\lambda\in Y_T(\RR)$,
then we define $P_\lambda$ to be the subgroup
generated by $T$ and the root groups $U_\alpha$,
where $\alpha$ ranges over all roots in $\Psi(G,T)$ such that
$\langle \lambda,\alpha\rangle\geq 0$; note that $P_\lambda$
is a parabolic subgroup of $G$, \cite[Prop.\ 8.4.5]{spr2}.
A Levi decomposition of $P_\lambda$ is given by
$P_\lambda = L_\lambda R_u(P_\lambda)$, where
$L_\lambda = C_G(\lambda)$ is the Levi subgroup of $P_\lambda$
generated by $T$ and the root groups $U_\alpha$ with
$\langle\lambda,\alpha\rangle = 0$.
The unipotent radical $R_u(P_\lambda)$ is generated by the root groups
$U_\alpha$,  where $\alpha$ ranges over all roots such that
$\langle \lambda,\alpha\rangle> 0$.
If $P$ is a parabolic subgroup of $G$ and $L$ is a Levi subgroup of $P$,
then there exists $\nu\in Y$ such that $P= P_\nu$ and $L= L_\nu$.

The space $Y(\QQ) = Y_G(\QQ)$ is made by glueing pieces $Y_T(\QQ)$.
We now construct a space $Y(\RR)$ from pieces $Y_T(\RR)$ in a similar way.
If $g\in G$ and $T$ is a maximal torus of $G$, then $g$ gives rise to a $\QQ$-linear map from $Y_T(\QQ)$ to $Y_{gTg^{-1}}(\QQ)$.
Hence $g$ gives rise to an $\RR$-linear map from $Y_T(\RR)$ to $Y_{gTg^{-1}}(\RR)$.
It follows that $G$ acts on the disjoint union $\bigcup_T Y_T(\RR)$.
Now we identify $\nu\in Y_T(\RR)$ with $x\cdot\nu\in Y_{xTx^{-1}}(\RR)$, for $x\in L_\nu$.
Then $Y(\RR)$ is the resulting quotient space.  Given $\lambda\in Y(\RR)$,
we define $P_\lambda$ and $L_\lambda$ in the obvious way.
We identify $Y(\QQ)$ with a subset of $Y(\RR)$; this embedding is
equivariant with respect to the actions of $G$ on $Y(\QQ)$ and $Y(\RR)$.

We also define $X(\QQ)$ and $X(\RR)$ as the disjoint union of pieces
$X_T(\QQ)$ and $X_T(\RR)$ as $T$ runs over the maximal tori of $G$.
The left action of $G$ on $Y$ (resp.\ $X$) extends to a left
action of $G$ on $Y(K)$ (resp.\ $X(K)$);
the pairings $\langle \ {,}\  \rangle$ between $Y_T$ and $X_T$
extend to give non-degenerate pairings
$Y_T(K) \times X_T(K) \to K$ for each maximal torus $T$ of $G$.
The norm $\left\|\,\right\|$ on $Y$ comes from integer-valued bilinear
forms on $Y_T$ for each maximal torus $T$ of $G$, by Definition \ref{def:norm}(b);
since each of these forms extends to a $K$-valued bilinear form on $Y_T(K)$, the norm
on $Y$ extends to a $G$-invariant norm on $Y(K)$, which we also denote by $\left\|\,\right\|$.
In particular, for any maximal torus $T$ of $G$, the subset
$Y_T(\RR)$ of $Y(\RR)$ is a real normed vector space,
and hence carries a natural topology coming from the norm.
We endow $Y_T(\QQ)$ with the relative topology coming from
the inclusion $Y_T(\QQ) \subset Y_T(\RR)$.


\begin{lem}
\label{lem:semicty}
Recall that $K = \QQ$ or $\RR$.
\begin{itemize}
\item[(i)] For any $\alpha \in X_T(K)$, the set of $\lambda \in Y_T(K)$ such that
$\langle \lambda,\alpha\rangle >0$ is open in $Y_T(K)$.
\item[(ii)] For any $\lambda\in Y_T(K)$, there is an open neighbourhood
$U$ of $\lambda$ in $Y_T(K)$ such that for any $\mu\in U$, we have $P_\mu\subseteq P_\lambda$.
\end{itemize}
\end{lem}

\begin{proof}
(i). This is clear: $\alpha$ defines an open half-space in $Y_T(K)$.

(ii). Choose $U$ to be the set of $\mu \in Y_T(K)$ such that whenever $\langle \lambda,\alpha\rangle >0 $
for a root $\alpha$, we have $\langle \mu,\alpha\rangle >0$ also.
By (i), $U$ is a finite intersection of open sets, so is open.
For $\mu \in U$, we then have $R_u(P_\mu) \supseteq R_u(P_\lambda)$.
It is a standard fact that this implies $P_\mu \subseteq P_\lambda$.
\end{proof}

\subsection{Convex cones}\label{sec:cvxcones}
Let $E$ be a finite-dimensional vector space over $K=\QQ$ or $\RR$;
in the former case we give $E$ the relative topology it
inherits from its embedding in $E\otimes_\QQ \RR$.
A subset $C$ of $E$ is called a \emph{cone} if it is closed under multiplication by
non-negative elements of $K$.
We recall some standard facts about cones;
for more detail, see for example the appendix and additional references in \cite{oda}.
A \emph{convex cone} in $E$ is a cone in $E$ which is also a convex subset.
Let $D\subseteq E^*$, where $E^*$ denotes the dual of $E$.
The set $\{e\in E \mid \beta(e)\geq 0 \mbox{ for all } \beta\in D\}$
is a closed convex cone in $E$;
we call this the \emph{cone defined by $D$}.
A convex cone $C$ is said to be \emph{polyhedral} if it is the cone defined
by some finite subset of $E^*$.
If $K=\QQ$ and $D$ is a finite subset of $E^*$,
then the subset of $E\otimes_\QQ \RR$ defined by $D$ is the closure of the subset of $E$ defined by $D$.

By the Minkowski-Weyl Theorem \cite{charnes},
a convex cone $C$ is polyhedral if and only if it is
\emph{finitely generated}: that is, if and only if
there exist $e_1,\ldots,e_s\in E$ for some $s$ such that
$C=\{c_1e_1+\cdots +c_se_s \mid c_1,\ldots,c_s\geq 0\}$
(we say that $C$ is the
\emph{cone generated by $e_1,\ldots,e_s$}).
In particular, a finitely generated convex cone is closed.
If $K=\QQ$ and $C$ is the finitely generated convex cone in $E$ generated by $e_1,\ldots,e_s\in E$,
then the cone in $E\otimes_\QQ \RR$ generated by $e_1,\ldots,e_s$ is the closure of $C$.  Likewise, if $D\subseteq E^*$ and $C$ is the finitely generated convex cone defined by $D$, then the cone in $E\otimes_\QQ \RR$ defined by $D$ is the closure of $C$.


\subsection{Vector buildings and spherical buildings}
\label{sub:buildings}
We derive our main results in this paper for subsets of $Y(K)$,
but we also wish to translate them into the language of spherical buildings.
In order to do this, we need to recall how to construct buildings from $Y(K)$.
Instead of moving straight from $Y(K)$ to the associated spherical building of $G$, we first pass to
the vector building and then identify the spherical building of $G$ as a subset of this vector building.
The additional structure afforded by the vector building makes the exposition more transparent.

First, define an equivalence relation on $Y(K)$ by $\lambda\equiv \mu$ if $\mu= u\cdot \lambda$ for some $u\in R_u(P_\lambda)$.
We let $V(K)= V_G(K)$ be the set of equivalence classes and let $\varphi\colon Y(K)\ra V(K)$ be the canonical projection
(to ease notation, we use $\varphi$ for the projection in both cases $K=\QQ$ and $K=\RR$).
We call $V(\RR)$ and $V(\QQ)$ the \emph{vector building of $G$} and \emph{rational vector building of $G$}, respectively,
see \cite[Sec.\ II, Sec.\ IV]{rousseau};
we have an obvious $G$-equivariant embedding from $V(\QQ)$ to $V(\RR)$.
Since the norm $\left\|\,\right\|$ on $Y(K)$ is $G$-invariant, it descends to give a real-valued function on $V(K)$, which we also call a norm and denote by $\left\|\,\right\|$.

Given a maximal torus $T$ of $G$, we set $V_T(K):= \varphi(Y_T(K))$.
We call the subsets $V_T(K)$ the \emph{apartments} of $V(K)$.
The restriction of $\varphi$ to $Y_T(K)$ is a bijection, so we can regard $V_T(K)$ as a vector space over $K$.
Any two points of $V(K)$ lie in a common apartment, because any two parabolic subgroups of $G$ contain a common maximal torus.
Because of this, we can put a metric $d$ on
$V(K)$ by defining $d(x,y) = \left\|x-y\right\|$ to be the Euclidean distance between $x$ and $y$ in any apartment that contains them both.
Similarly, we let $[x,y]$ denote the line segment between $x$ and $y$ in any apartment containing them both.
Neither of these constructions depends on the choice of apartment (\cite[Sec.\ II]{rousseau}).  Likewise, if $a,b\in K$, then the linear combination $ax+by$ does not depend on the choice of apartment.
By \cite[Prop.\ 2.3]{rousseau}, $V(\RR)$ is a complete geodesic metric space; it is the completion of the space $V(\QQ)$
with respect to the norm.

If $W\subseteq V(K)$ and $T$ is a maximal torus of $G$, then we define $W_T:= W\cap V_T(K)$.
We say that $W$ is \emph{convex} if $W$ contains the line segment $[x,y]$ for all $x,y \in W$.
If $W$ is convex, then $W_T$ is a convex subset of $V_T(K)$ for every maximal torus $T$ of $G$, and vice versa.

Now the \emph{spherical Tits building $\Delta(\RR) = \Delta_G(\RR)$ of $G$} can be defined simply as the unit sphere in $V(\RR)$,
and the \emph{rational spherical building $\Delta(\QQ) = \Delta_G(\QQ)$ of $G$} is the projection of $V(\QQ)\setminus\{0\}$ onto $\Delta(\RR)$,
\cite[IV]{rousseau}, \cite[Ch.\ 2,$\S$2]{mumford}.
We have an obvious $G$-equivariant embedding from $V(\QQ)$ to $V(\RR)$.
Since the norm on $V(K)$ is $G$-invariant, $\Delta(K)$ is a $G$-invariant subspace of $V(K)$.
It is clear that $\Delta(K)$ is a closed subspace of $V(K)$ and that
the metric on $V(K)$ restricts to give a metric on $\Delta(K)$,
\cite[II]{rousseau}; since we are working with vectors of norm $1$ in $V(K)$,
this metric gives the same topology on $\Delta(K)$ as that coming from the angular metric defined in \cite[Ch.\ 2, $\S$2, p.\ 59]{mumford}.
In particular, $\Delta(\RR)$ is the completion of $\Delta(\QQ)$.
There is a natural notion of opposition of points in $\Delta(K)$ inherited from $V(K)$;
$x$ and $y$ are \emph{opposite} if and only if $d(x,y) = 2$.
Given any two points in $\Delta(K)$ that are not opposite, there is a unique geodesic between them;
this is the projection of the corresponding line segment in $V(K)$ onto the unit sphere.
We define the \emph{apartments} of $\Delta(K)$ to be the intersections of the apartments of $V(K)$ with $\Delta(K)$;
we set $\Delta_T(K):= \Delta(K)\cap V_T(K)$.
Each apartment $\Delta_T(K)$ is the unit sphere centred at the origin in the Euclidean space $V_T(K)$.
We denote the projection map from $V(K)\setminus\{0\}$ to $\Delta(K)$ by $\xi$ and we
define
$$
\zeta\colon Y(K)\setminus\{0\}\ra \Delta(K)
$$
to be the composition $\xi\circ \varphi$
(again, here we use the same letter for these maps in both cases $K=\QQ$ and $K=\RR$).

If $\Sigma\subseteq \Delta(K)$, then we define $\Sigma_T:= \Sigma\cap \Delta_T(K)$.
We say that $\Sigma$ is \emph{convex} if whenever $x,y\in \Sigma$ are not opposite, then $\Sigma$
contains the geodesic between $x$ and $y$, \cite[\S 2.1]{serre2}.
It follows that $\Sigma$ is convex if and only if $\Sigma_T$ is a convex subset of $\Delta_T(K)$ for every maximal torus $T$ of $G$.

The spherical building $\Delta(K)$ has a simplicial structure.
It is the geometric realization over $K$ of an abstract building, whose
simplices correspond to the parabolic subgroups of $G$ (ordered by reverse inclusion).
In our notation, given a proper parabolic subgroup $P$ of $G$, we can recover the corresponding
topological simplex
as $\sigma_P= \{\zeta(\lambda)\mid \lambda\in Y(K), P\subseteq P_\lambda\}$.
Since $\left\|\,\right\|$ is $G$-invariant, the action of $G$ on $V(K)$ restricts to give an action of $G$ on $\Delta(K)$ by isometries;
this action preserves the simplicial structure.  Note that $\zeta, \xi$ and $\varphi$ are $G$-equivariant.
For any $\lambda \in Y(K)\setminus\{0\}$, the stabilizer in $G$ of
$\varphi(\lambda) \in V(K)$ and the stabilizer in $G$ of $\zeta(\lambda) \in \Delta(K)$ are both equal to
$P_\lambda$.

\subsection{Cones in $Y(K)$}
In this paper, we wish to move back and forth between $Y(K)$ and the building $\Delta(K)$, using the map $\zeta\colon Y(K)\setminus\{0\}\ra \Delta(K)$.
In particular, we study what happens to convex subsets of spherical buildings when we pull them back to $Y(K)$.
This leads to the following basic definitions:

\begin{defn}
\label{def:convex-saturated}
Given a subset $C$ of $Y(K)$ and a maximal torus $T$ of $G$, we set $C_T := C\cap Y_T(K)$.
\begin{itemize}
\item[(i)] We say that $C$ is \emph{convex} if $C_T$ is a
convex subset of $Y_T(K)$ for every maximal torus $T$ of $G$.
\item[(ii)]
We say that $C$ is \emph{saturated}
if whenever $\lambda \in C$, then $u\cdot \lambda \in C$ for all
$u \in R_u(P_\lambda)$.
\item[(iii)]
We say that $C$ is a \emph{cone} if $C_T$ is a cone
for every maximal torus $T$ of $G$.
In this case we say that $C$ is \emph{polyhedral}
if every $C_T$ is polyhedral,
and that $C$ is of \emph{finite type} if for every $T$,
the set $\{g\cdot(C_{g\inverse Tg})\mid g\in G\}$ is finite.
\end{itemize}
\end{defn}

\begin{defn}
\label{def:polyhedralsubset}
Let $\Sigma$ be a  convex subset of $\Delta(K)$
and let $C= \zeta^{-1}(\Sigma) \cup \{0\}$.
From the definition of $\zeta$, it is clear that $C$ is a saturated cone in $Y(K)$.
We say that $\Sigma$ is \emph{polyhedral}
if $C$ is polyhedral, and in this case we
say $\Sigma$ is \emph{of finite type} if $C$ is of finite type.
\end{defn}


The next lemma shows how these definitions allow us to translate back and forth between $Y(K)$ and $\Delta(K)$.

\begin{lem}
\label{lem:Ycones}
Let $\Sigma$ be a subset of $\Delta(K)$ and let $C$ be any saturated cone in $Y(K)$ such that $\zeta(C\setminus\{0\}) = \Sigma$.
Then the following hold:
 \begin{itemize}
  \item[(i)] $C = \zeta^{-1}(\Sigma) \cup \{0\}$;
  \item[(ii)] $\Sigma$ is convex if and only if $C$ is convex;
  \item[(iii)] $\Sigma$ is polyhedral if and only if $C$ is polyhedral;
  \item[(iv)] $\Sigma$ is of finite type if and only if $C$ is of finite type.
 \end{itemize}
\end{lem}

\begin{proof}
(i). Since $\zeta(C\setminus\{0\}) = \Sigma$, we have $C \subseteq \zeta^{-1}(\Sigma) \cup \{0\}$.
On the other hand, suppose $\lambda \in \zeta^{-1}(\Sigma)$.
Then there exists $\mu \in C$ such that $\zeta(\mu) = \zeta(\lambda)$.
By definition of $\zeta$, this implies that there exists $u \in R_u(P_\mu)$ such that
$\lambda$ is a positive multiple of $u\cdot\mu$.
But $C$ is a saturated cone, so we must have $\lambda \in C$.

Now (ii), (iii) and (iv) follow from the definitions.
\end{proof}

We now show that subcomplexes of the building fit into this framework.

\begin{lem}
\label{lem:cxcvx}
Let $\Sigma$ be a convex subcomplex of $\Delta(K)$.
Then $\Sigma$ is closed, polyhedral, and of finite type.
\end{lem}

\begin{proof}
It is standard that a convex subcomplex is closed in $\Delta(K)$.
Let $C = \zeta^{-1}(\Sigma) \cup \{0\}$;
then $C$ is a saturated convex cone in $Y(K)$.
Let $P$ be a parabolic subgroup of $G$ such that
$\sigma_P \in \Sigma$.
Let $T$ be a maximal torus of $P$ and let
$B$ be a Borel subgroup of $P$ with $B\supseteq T$.
Let $\Pi=\{\alpha_1,\ldots,\alpha_r\}$ be the
base for the root system $\Psi(G,T)$ corresponding to $B$.
Then $\Pi$ is a basis for the space $X_T(K)$.
Let $\{\lambda_1,\ldots,\lambda_r\}$ denote the
corresponding dual basis of $Y_T(K)$,
so that $\langle \lambda_i,\alpha_j\rangle =\delta_{ij}$
for $1\leq i\leq r$.

Now there exists a subset $\Pi'\subseteq \Pi$
such that $P$ is of the form $P(\Pi')$ in the notation of \cite[IV, 14.17]{Bo}
(this means that the Levi subgroup of $P$ containing $T$ has root system
spanned by the subset $\Pi'$, and the unipotent radical of $P$
contains all the root groups $U_\alpha$ with $\alpha \in \Pi\setminus\Pi'$).
Now for any $\lambda\in Y_T(K)$, $P\subseteq P_\lambda$
if and only if $\langle \lambda,\alpha_i\rangle=0$
for $\alpha_i\in \Pi'$ and $\langle \lambda,\alpha_i\rangle\geq 0$ for
$\alpha_i\not\in \Pi'$.

Let $C^P = \zeta\inverse(\sigma_P) \cup \{0\}$.
Then $C^P_T = C^P \cap Y_T(K)$ consists of all the
$\lambda \in Y_T(K)$ such that $P \subseteq P_\lambda$.
We claim that $C^P_T$ is the cone in $Y_T(K)$ generated by the
set $\{\lambda_j\mid \alpha_j \notin \Pi'\}$.
This follows easily from the characterisation of parabolic
subgroups containing $P$ given in the previous paragraph.
This shows that $C^P_T$ is a finitely generated
cone in $Y_T(K)$ for every maximal torus $T$ of $G$ contained in $P$.

Now suppose $T$ is a maximal torus of $G$ not contained in $P$,
and let $Q$ denote the subgroup of $G$ generated by $P$ and $T$.
Since $Q$ contains $P$, $Q$ is also a parabolic subgroup of $G$,
and $C^P_T = C^P\cap Y_T(K) = C^Q_T$ is a finitely generated cone
in $Y_T(K)$ by the above arguments applied to $Q$.
We have now shown that $C^P_T$ is a finitely generated cone in
$Y_T(K)$ for every maximal torus $T$ of $G$.

It is clear that $C_T = C \cap Y_T(K)$ is the union of the cones $C^P_T$ as $P$ runs over
the parabolic subgroups of $G$ containing $T$ with $\sigma_P \in \Sigma$.
Since there are only finitely many parabolic subgroups of $G$ containing any given
maximal torus and each $C_T^P$ is finitely generated, we can conclude that because $C_T$ is convex,
$C_T$ is also finitely generated.
Hence, by the Minkowski-Weyl Theorem, $C_T$ is polyhedral for each $T$,
and hence $C$ is polyhedral.

It remains to show that $C$ is of finite type.
This also follows from the fact that for any maximal torus $T$ of $G$,
the set of parabolic subgroups of $G$ that contain $T$ is finite.
So there are only finitely many possibilities for $g\cdot(C_{g\inverse Tg})$
as $g$ ranges over all the elements of $G$.
Hence $C$ is of finite type, as required.
\end{proof}

\subsection{Tits' Centre Conjecture}
\label{subsec:tcc}

Suppose $\Sigma$ is a closed convex subset of $\Delta(\RR)$.
If there exists a point of $\Sigma$ which has no opposite in $\Sigma$, then $\Sigma$ is \emph{contractible}:
that is, $\Sigma$ has the homotopy type of a point.
The converse is also true: if every point of $\Sigma$ has an opposite in $\Sigma$, then $\Sigma$ is not contractible.
For these results, and further characterizations of contractibility, see \cite[Thm.\ 1.1]{BL06}, and also
\cite[\S 2.2]{serre2}.
This dichotomy leads to the following definitions, where our terminology is motivated by that of Serre \cite[Def.\ 2.2.1]{serre2}:

\begin{defn}
Recall that $K = \QQ$ or $\RR$.

\begin{itemize}
\item[(i)] Let $\Sigma$ be a convex subset of $\Delta(K)$.
We say that $\Sigma$ is \emph{$\Delta(K)$-completely reducible} (or $\Delta(K)$-cr) if every point in $\Sigma$ has an opposite in $\Sigma$.
\item[(ii)] Let $C$ be a convex, saturated cone in $Y(K)$.
We say that $C$ is \emph{$Y(K)$-completely reducible} (or $Y(K)$-cr) if for every $\lambda\in C$,
there exists $u \in R_u(P_\lambda)$ such that $-(u\cdot\lambda)\in C$
(note that it is automatic that $u\cdot\lambda \in C$ for all $u \in R_u(P_\lambda)$, since $C$ is saturated).
\end{itemize}
\end{defn}

Recall from Definition \ref{def:polyhedralsubset} that $C:=  \zeta^{-1}(\Sigma) \cup \{0\}$
is a saturated convex cone;
it is immediate that $\Sigma$ is
$\Delta(K)$-cr if and only if $C$ is $Y(K)$-cr.

\begin{defn}
\label{defn:centre}
Let $\Sigma$ be a subset of $\Delta(K)$ and let $c\in \Sigma$.
Let $\Gamma$ be a group acting on $\Delta(K)$ by building automorphisms.
We say that $c$ is a \emph{$\Gamma$-centre of $\Sigma$} if $c$ is fixed by $N_\Gamma(\Sigma)$.
\end{defn}



The following is a version of the so-called ``Centre Conjecture'' by J.\ Tits,
cf.\  \cite[Lem.\ 1.2]{tits0}, \cite[\S 4]{serre1.5},  \cite[\S 2.4]{serre2}, \cite{tits2}, \cite[Ch.\ 2, $\S$3]{mumford},
\cite[Conj.\ 3.3]{rousseau}, \cite{muhlherrtits}, \cite{lrc}, \cite{rc}.

\begin{conj}
\label{conj:tcc}
Let $\Sigma$ be a closed convex 
subset of $\Delta(K)$.
Then at least one of the following holds:
\begin{itemize}
\item[(i)] $\Sigma$ is $\Delta(K)$-cr;
\item[(ii)] $\Sigma$ has an $\Aut \Delta(K)$-centre.
\end{itemize}
\end{conj}

Conjecture \ref{conj:tcc} often appears in the literature with the assumption that $\Sigma$ is
a \emph{subcomplex} of $\Delta(K)$, rather than an arbitrary closed convex subset.
In this form, the conjecture is known;
this is the culmination of work of B.\ M\"uhlherr and J.\ Tits \cite{muhlherrtits}
($G$ of classical type or type $G_2$),
B.\ Leeb and C.\ Ramos-Cuevas \cite{lrc}
($G$ of type $F_4$ or $E_6$) and C.\ Ramos-Cuevas \cite{rc}
($G$ of type $E_7$ or $E_8$).

\begin{thm}
\label{thm:knowntcc}
If $\Sigma$ is a convex subcomplex of $\Delta(K)$,
then Conjecture \ref{conj:tcc} holds.
\end{thm}

When $\Sigma$ is a closed convex subset of $\Delta(K)$ but not a subcomplex, very few cases of Conjecture \ref{conj:tcc} are known.
If the dimension of $\Sigma$ is at most $2$, then the conjecture is true, \cite{BL05}.

The proofs of the various cases of Theorem \ref{thm:knowntcc} in \cite{muhlherrtits}, \cite{lrc} and \cite{rc}
rely on the extra simplicial structure carried by a subcomplex, and it is not clear whether these methods
can be extended to arbitrary convex subsets of $\Delta(K)$, \cite[Sec.\ 1]{rc}.
One area in which relevant cases of the Centre Conjecture have been known for some time is Geometric Invariant Theory,
see \cite{kempf}, \cite{rousseau}, \cite{mumford}.
It is our intention in this paper to elucidate and extend the methods of {\em loc.\ cit.}\ with particular reference to Conjecture \ref{conj:tcc}.
In doing this, we are able to consider a wider class of convex subsets of a spherical building $\Delta(K)$ and show how to reformulate
the Centre Conjecture for a subset of this class.
This class consists precisely of the convex subsets
of $\Delta(K)$ that are polyhedral and of finite type, cf. Definition \ref{def:polyhedralsubset}.

\begin{rem}
It is worth pointing out that in Conjecture \ref{conj:tcc} the subset $\Sigma$ is assumed to be closed
in $\Delta(K)$, whereas in most of our results in the sequel we do not require this hypothesis of closedness.
Thus, in some sense,
we are looking at a slightly generalized version of the conjecture.
However, we do need to impose the extra conditions that $\Sigma$ is polyhedral and of finite type, and we
restrict attention in this paper to finding $G$-centres, rather than $\Aut \Delta(K)$-centres,
so this narrows the field again.
Note that a convex \emph{subcomplex} of a spherical building, being both closed and polyhedral
of finite type by Lemma \ref{lem:cxcvx} above, fits into either camp.
\end{rem}



\subsection{$G$-complete reducibility}
\label{sec:gcr}

We briefly recall some definitions and results concerning Serre's notion
of $G$-complete reducibility for subgroups of $G$, see
\cite{serre1.5}, \cite{serre2},
\cite{BMR}, \cite{BMRT}, and \cite{GIT} for more details.
A subgroup $H$ of $G$ is called \emph{$G$-completely reducible} ($G$-cr) if
whenever $H$ is contained in a parabolic subgroup $P$ of $G$,
there exists a Levi subgroup of $P$ containing $H$.
This concept can be interpreted in the building
$\Delta(K)$ of $G$:
let
$$
\Delta(K)^H := \bigcup_{H \subseteq P} \sigma_P,
$$
the fixed point set of $H$ in $\Delta(K)$.  Then $\Delta(K)^H$ is a convex subcomplex of $\Delta(K)$, which is $\Delta(K)$-cr if and only if $H$ is $G$-cr \cite[\S3]{serre2}.
The study of $G$-complete reducibility motivated much of the work in this paper;
for a direct application of (the known cases of) the Centre Conjecture \ref{conj:tcc}
to $G$-complete reducibility, see \cite{BMR:sepext}.

In the proof of Theorem \ref{thm:tcc} below, we require a piece of terminology introduced in \cite[Def.\ 5.4]{GIT}.
Let $H$ be a subgroup of $G$, let $G \to \GL_m$ be an embedding of algebraic groups, and let $n \in \NN$.
We call $\mathbf{h} \in H^n$ a \emph{generic tuple of $H$} if the components of $\mathbf{h}$ generate the
associative subalgebra of $\Mat_m$ spanned by $H$.
Generic tuples always exist if $n$ is sufficiently large.
Now $G$ acts on $G^n$ by simultaneous conjugation, and $H$ is $G$-cr if and only if the
$G$-orbit of $\mathbf{h} \in H^n$ is closed in $G^n$,
where $\mathbf{h}$ is a generic tuple of $H$, \cite[Thm.\ 5.8(iii)]{GIT}.

\subsection{Instability in GIT}
\label{sec:instab}

Many of the results in this paper are inspired by
constructions of Kempf \cite{kempf} and Hesselink \cite{He}, and
also by our generalization of their work in \cite{GIT}.
We briefly recall some of the main definitions which are relevant to our subsequent discussion
(see especially Section \ref{sec:gitstcc} below).
Throughout this section,
$G$ acts on an affine variety $A$, and $S$ is a non-empty $G$-stable closed subvariety of $A$.
We denote the $G$-orbit of $x$ in $A$ by $G\cdot x$.
For any $\lambda \in Y$, there is a morphism
$\phi_{x,\lambda}:k^* \to A$,
given by $\phi_{x,\lambda}(a) = \lambda(a)\cdot x$ for each $a \in k^*$.
If this morphism extends to a
morphism $\widehat{\phi}_{x,\lambda}:k \to A$, then we say that
$\lim_{a\to 0} \lambda(a)\cdot x$ \emph{exists}, and we set this limit
equal to $\widehat{\phi}_{x,\lambda}(0)$.
In this case we say that $\lambda$ \emph{destabilizes $x$}, and we say that
$\lambda$ \emph{properly destabilizes $x$} if the limit does not belong to $G\cdot x$;
we call the corresponding parabolic subgroup $P_\lambda$ a \emph{(properly) destabilizing parabolic subgroup for $x$}.

The following are \cite[Def.\ 4.2 and Def.\ 4.4]{GIT}.

\begin{defn}
\label{def:destabilizing}
For each non-empty subset $U$ of $A$,
define $|A,U|$ as the set of
$\lambda \in Y$ such that $\underset{a\to 0}{\lim}\, \lambda(a)\cdot x$
exists for all $x\in U$.
We define
$$ |A,U|_S= \{\lambda\in |A,U|\mid \underset{a\to 0}{\lim}\, \lambda(a)\cdot x \in S\ \mbox{for all $x \in U$}\}. $$
If $\lambda\in |A,U|_S$, then we say $\lambda$ \emph{destabilizes $U$ into $S$}
or is a \emph{destabilizing cocharacter for $U$ with respect to $S$}.
Extending Hesselink \cite{He}, we call $U$ \emph{uniformly $S$-unstable} if $|A,U|_S$ is non-empty;
if, in addition, $U \not\subseteq S$, we call $U$ \emph{properly uniformly $S$-unstable}.
We write $|A,U|_s$ instead of $|A,U|_{\{s\}}$ if $S= \{s\}$ is a singleton, and we write $|A,x|_S$ instead of $|A,\{x\}|_S$ if $U= \{x\}$ is a singleton.
By the
Hilbert-Mumford Theorem \cite[Thm.\ 1.4]{kempf}, $x\in A$ is $S$-unstable
if and only if $\overline{G\cdot x}\cap S\ne\varnothing$.
Note that if $U$ is properly uniformly $S$-unstable, then $|A,U|_S$ is a proper subset of $|A,U|$
(we could have $\lambda\in |A,U|$ such that $\lim_{a\ra 0} \lambda(a)\cdot x$ does not lie in $S$
for some $x \in U$; the zero cocharacter gives an example of this phenomenon).
\end{defn}

Suppose $U$ is properly uniformly $S$-unstable.
In \cite[Thm.\ 4.5]{GIT}, we constructed a so-called \emph{optimal class} of cocharacters contained
in $|A,U|_S$ which enjoys a number of useful properties.
That construction consisted of a strengthening of arguments of Kempf and Hesselink;
one of the main goals of this paper is to extend these ideas even further and interpret them in
the language of buildings, where they give new positive results for Conjecture \ref{conj:tcc}.
The central idea is that the subset $|A,U|$ of $Y$ gives rise to a
convex subset of the building $\Delta(K)$,
and if $\lambda$ belongs to the optimal class, then $\zeta(\lambda)$ is a $G$-centre of this subset;
see Section \ref{sec:gitstcc} for precise details.
Indeed, many of the results of Kempf \cite{kempf} and Hesselink \cite{He}, and our uniform $S$-instability results in \cite[Sec.\ 4]{GIT},
can be recovered as special cases of the general constructions presented in this paper; see for
example Remark \ref{rem:getkempf} below.

\section{Quasi-states and optimality}
\label{sec:quasistates}

In this section we generalize some of the results of Kempf from
\cite{kempf}, concerning states.
We then translate these results into the language
of buildings, and show how they can be used to
prove Conjecture \ref{conj:tcc} in various cases.
Our core result is Theorem \ref{thm:kempf2.2},
which generalizes Kempf's key theorem \cite[Thm.\ 2.2]{kempf}.

The main point of the material at the start of this section is
that many results in \cite{kempf} go through under considerably
weaker hypotheses;
this allows us to extend Kempf's formalism to cover important new cases,
as demonstrated in Sections \ref{sec:buildings} and \ref{sec:gitstcc}.
We start by introducing quasi-states, generalizing Kempf's notion of a state \cite[Sec.~2]{kempf}.

\begin{defn}
\label{def:quasistate}
A \emph{real quasi-state} $\Xi$ of $G$ is an
assignment of a finite (possibly empty) set $\Xi(T)$ of elements
of $X_T(\RR)$ for each maximal torus $T$ of $G$.
If $\Xi(T)\subseteq X_T(\QQ)$ for every $T$,
then we call $\Xi$ a \emph{rational quasi-state},
and if $\Xi(T)\subseteq X_T$ for every $T$,
then we call $\Xi$ an \emph{integral quasi-state}.
If $K= \QQ$ or $\RR$, then \emph{$K$-quasi-state} has the obvious meaning.
Given a real quasi-state $\Xi$ and $g\in G$, we define a new
real quasi-state $g_* \Xi$ by
\[
(g_* \Xi)(T) := g_! \Xi(g^{-1}Tg) \subseteq X_T(K).
\]
This defines a left action of $G$ on the set of real quasi-states
of $G$.
Note that if $\Xi$ is rational (resp.\ integral),
then so is $g_*\Xi$ for any $g\in G$.
For each real quasi-state $\Xi$, we write
$C_G(\Xi) = \{g \in G \mid g_* \Xi = \Xi\}$ for
the centralizer of $\Xi$ in $G$.
We say that
$\Xi$ is \emph{bounded} if for every maximal torus
$T$ of $G$, the set $\bigcup_{g\in G}(g_* \Xi)(T)$ is finite.
Note that if $\Xi$ is a bounded $\QQ$-quasi-state, then there exists some $n \in \NN$
such that for all maximal tori $T$ of $G$ and any $\alpha \in \Xi(T)$, we have
$n\alpha \in X_T$;
i.e., scaling a rational quasi-state by a suitably large positive integer, we can ensure it becomes integral.
\end{defn}

\begin{defn}
\label{defn:numer}
Associated to a real quasi-state $\Xi$ and a maximal torus $T$ of $G$,
we have the $\RR$-valued function
$\mu(\Xi,T,\cdot)\colon Y_T(\RR)\ra \RR$ defined by
\[
\mu(\Xi,T,\lambda)
:= \min_{\alpha\in \Xi(T)} \langle \lambda,\alpha \rangle.
\]
We call $\mu(\Xi,T,\cdot)$ the \emph{numerical function of $\Xi$ and $T$}.
Note that for a fixed maximal torus $T$ of $G$,
$\Xi(T)$ is empty if and only if $\mu(\Xi,T,\lambda) = \infty$ for some
$\lambda \in Y_T(\RR)$ if and only if $\mu(\Xi,T,\lambda) = \infty$ for all $\lambda \in Y_T(\RR)$.
Note also that if $\Xi$ is rational (resp.\ integral),
then the associated numerical function takes rational (resp.\ integer)
values on $Y_T(\QQ)$ (resp.\ $Y_T)$, wherever it is finite.

Suppose $\lambda \in Y(\RR)$.
We say that $\Xi$ is \emph{admissible at $\lambda$}
if for any maximal torus $T$ of $G$ with $\lambda \in Y_T(\RR)$ and any $x\in P_\lambda$,
we have
\[
\mu(\Xi,xTx^{-1},x \cdot \lambda)= \mu(\Xi,T, \lambda).
\]
By extension, for any subset $S$ of $Y(\RR)$, we say that $\Xi$ is \emph{admissible on $S$}
if $\Xi$ is admissible at every point of $S$.
If $\Xi$ is admissible on all of $Y(\RR)$,
then we simply call $\Xi$ an \emph{admissible} quasi-state
(note that this agrees with the definition of admissibility given in \cite[Sec.\ 2]{kempf}).

We say that $\Xi$ is \emph{quasi-admissible}
if for any maximal torus $T$ of $G$, any $\lambda\in Y_T(\RR)$,
and any $x\in P_\lambda$,
we have

\[
\mu(\Xi,T,\lambda)\geq 0 \quad \Rightarrow \quad
\mu(\Xi,xTx^{-1},x\cdot \lambda)\geq 0.
\]
Note that if $\Xi$ is admissible,  
then $\Xi$ is quasi-admissible.
\end{defn}

\begin{rem}
\label{rem:kempforus}
Our concept of a quasi-state is weaker than Kempf's notion of a state \cite[Sec.\ 2]{kempf}.
However, Kempf's main result \cite[Thm.~2.2]{kempf} goes through
with his bounded admissible states
replaced by bounded admissible quasi-states. 
The real difference between our results and Kempf's is that
we replace admissibility with the weaker notions of
admissibility at a point and quasi-admissibility;
this gives us genuinely new results.

We also note that it is rather important for our purpose of translating results
into the language of buildings to be able to use admissibility at a point,
rather than Kempf's stronger notion of admissibility;
see especially Theorem \ref{thm:partialTCC} and the subsequent Remark
\ref{rems:howtoprovetcc}(ii) below.
\end{rem}



We collect some useful properties of quasi-states in the next lemma.

\begin{lem}
\label{lem:propertiesofstates}
Suppose $\Xi$ is a real quasi-state of $G$.
\begin{itemize}
\item[(i)] Suppose $T$ is any maximal torus of $G$,
and suppose that $\lambda_1,\lambda_2 \in Y_T(\RR)$ are such that
$\mu(\Xi,T,\lambda_i) > 0$ for $i=1,2$.
Then $\mu(\Xi,T,\nu) > 0$ for any $\nu = a_1\lambda_1 + a_2\lambda_2$, where $a_1, a_2 \in \RR_{\geq 0}$
are not both $0$.
\item[(ii)] For any maximal torus $T$ of $G$, any $\lambda \in Y_T(\RR)$, and any $g \in G$, we have
$$
\mu(\Xi,T,\lambda) = \mu(g_*\Xi, gTg\inverse, g\cdot\lambda).
$$
\item[(iii)] If $\Xi$ is admissible at $\lambda \in Y(\RR)$, then
the value of $\mu(\Xi,T,\lambda)$ is independent of the choice of maximal torus
$T$ with $\lambda\in Y_T(\RR)$.
\item[(iv)] If $\Xi$ is quasi-admissible and $\lambda \in Y(\RR)$, then
whether $\mu(\Xi,T,\lambda)$ is non-negative or not
is independent of the choice of maximal torus
$T$ with $\lambda\in Y_T(\RR)$.
\item[(v)] If $\Xi$ is admissible at $\lambda \in Y(\RR)$, then $g_*\Xi$
is admissible at $g\cdot\lambda$ for any $g \in G$.
\item[(vi)]If $\Xi$ is quasi-admissible, then $g_*\Xi$ is quasi-admissible for any $g \in G$.
\end{itemize}
\end{lem}

\begin{proof}
(i). For $i=1,2$, we have $\mu(\Xi,T,\lambda_i) >0$ if and only if $\langle\lambda_i,\alpha\rangle >0$ for all $\alpha \in \Xi(T)$.
If this holds, and if $a_1,a_2 \in \RR_{\geq 0}$ are not both $0$, then $\langle a_1\lambda_1 + a_2\lambda_2,\alpha\rangle >0$ for all $\alpha \in \Xi(T)$.
This gives the result.

(ii). By definition, $(g_*\Xi)(gTg\inverse) = g_!\Xi(T)$.
The result now follows from \eqref{eqn:ipinvce}.

(iii) and (iv). Suppose $T$ and $T'$ are maximal tori of $G$ and
$\lambda\in Y_T(\RR)\cap Y_{T'}(\RR)$.
Then $xTx^{-1}=T'$ for some
$x\in L_\lambda \subseteq P_\lambda$.
So if $\Xi$ is admissible at $\lambda$, we have
$$
\mu(\Xi,T,\lambda) = \mu(\Xi,xTx^{-1},x\cdot \lambda) = \mu(\Xi,T',\lambda),
$$

which proves (iii).
Similarly, if $\Xi$ is quasi-admissible, then
$$
\mu(\Xi,T,\lambda)\geq 0 \iff \mu(\Xi,xTx^{-1},x\cdot \lambda)\geq 0
\iff \mu(\Xi,T',\lambda)\geq 0,
$$
which proves (iv).

(v) and (vi). Suppose $g \in G$, $\lambda \in Y_T(\RR)$ and $x \in P_\lambda$.
Set $\lambda' = g\cdot\lambda$, $T' = gTg\inverse$ and $y = gxg\inverse$.
Then $\lambda' \in Y_{T'}(\RR)$ and $y \in P_{\lambda'}$.
By part (ii), in this situation we have
\begin{equation}\label{eqn:elephant}
\mu(g_*\Xi, T', \lambda') = \mu(\Xi,T,\lambda).
\end{equation}
Moreover, for the same reason, we also have
\begin{equation}\label{eqn:trousers}
\mu(g_*\Xi,yT'y\inverse,y\cdot\lambda')
=
\mu(\Xi,xTx\inverse,x\cdot\lambda).
\end{equation}
Now suppose $\Xi$ is admissible at $\lambda$.
Then $\mu(\Xi,T,\lambda) = \mu(\Xi,xTx\inverse,x\cdot\lambda)$.
Combining \eqref{eqn:elephant} and \eqref{eqn:trousers} shows that $g_*\Xi$ is admissible at $\lambda'$,
which proves (v).

Finally, suppose $\Xi$ is quasi-admissible.
Then by \eqref{eqn:elephant},
we have $\mu(g_*\Xi,T',\lambda') \geq 0$ if and only if $\mu(\Xi,T,\lambda) \geq 0$,
so $g_*\Xi$ is also quasi-admissible, by \eqref{eqn:trousers}, which proves (vi).
\end{proof}

\begin{defn}
\label{def:Z}
If $\Xi$ is a quasi-admissible $K$-quasi-state, then we can define a subset $Z(\Xi)$
of $Y(K)$ by setting
\[
Z(\Xi):= \{\lambda\in Y(K)\mid \exists
\text{ a maximal torus $T$ of $G$ with } \lambda\in Y_T(K)
\text{ and }\mu(\Xi,T,\lambda)\geq 0 \}.
\]
Here we use the convention that $\infty > 0$, so that in particular if $\Xi(T)= \varnothing$ for every maximal torus $T$ of $G$,
then $Z(\Xi) = Y(K)$.
By Lemma \ref{lem:propertiesofstates}(iv),
the quasi-admissibility of $\Xi$ implies that whether or not $\lambda$ belongs to $Z(\Xi)$ is independent
of which maximal torus $T$ of $G$ we choose with $\lambda \in Y_T(K)$.
\end{defn}

Our next two results show how to make new quasi-states by taking unions,
and how to exert some control over the stabilizer of a quasi-state.

\begin{lem}\label{lem:unionofstates}
Let $I$ be an arbitrary indexing set.
For each $i \in I$, let $\Xi_i$ be a $K$-quasi-state.
We define $\Xi := \bigcup_{i\in I}\Xi_i$ by setting $\Xi(T):= \bigcup_{i \in I} \Xi_i(T)$ for each maximal torus $T$ of $G$. Then:
\begin{itemize}
\item[(i)] If $\Xi(T)$ is finite for all maximal tori $T$ of $G$,
then $\Xi$ is a $K$-quasi-state.
\item[(ii)] If for some maximal torus $T$ of $G$ (and hence for every maximal torus $T$ of $G$),
$$
\bigcup_{i\in I} \left(\bigcup_{g\in G}(g_*\Xi_i)(T)\right)
$$
is finite, then $\Xi$ is a bounded $K$-quasi-state.
\item[(iii)] If $\Xi$ is a $K$-quasi-state and
every $\Xi_i$ is admissible at $\lambda \in Y(\RR)$, then
$\Xi$ is admissible at $\lambda$.
Similarly, if each $\Xi_i$ is quasi-admissible, then so is $\Xi$.
\item[(iv)] Suppose $\Xi$ is a $K$-quasi-state.  If $T$ is a maximal torus of $G$, $\lambda\in Y_T(K)$ and $\mu(\Xi_i,T,\lambda)>0$ for all $i$, then $\mu(\Xi,T,\lambda)>0$.
\item[(v)] Suppose $\Xi$ is a $K$-quasi-state.  Let $T$ be a maximal torus of $G$ and suppose that for every $i\in I$, there exists $\lambda_i\in Z(\Xi)_T$ such that $\mu(\Xi_i,T,\lambda_i)>0$.  Then there exists $\gamma\in Z(\Xi)_T$ such that $\mu(\Xi,T,\gamma)>0$.
\end{itemize}
\end{lem}

\begin{proof}
(i) and (ii) are immediate. For (iii), suppose each $\Xi_i$ is admissible at $\lambda$,
and choose a maximal torus $T$ of $G$ such that
$\lambda \in Y_T(K)$.
Then for $x \in P_\lambda$, we can write
$$
\mu(\Xi,xTx\inverse,x\cdot\lambda) = \min_{i\in I} \mu(\Xi_i,xTx\inverse,x\cdot\lambda) = \min_{i\in I} \mu(\Xi_i,T,\lambda) = \mu(\Xi,T,\lambda),
$$
where the admissibility of each $\Xi_i$ tells us that
$\mu(\Xi_i,xTx\inverse,x\cdot\lambda) = \mu(\Xi_i,T,\lambda)$ for each $i$,
and (i) implies that to calculate each ``$\min$'' we only need to consider a finite set of these values.  This proves (iii).  For (iv), note that $\Xi(T)= \bigcup_{i\in J} \Xi_i(T)$ for some finite subset $J$ of $I$, so we have $\mu(\Xi,T,\lambda)= {\rm min}_{i\in J} \mu(\Xi_i,T,\lambda)> 0$.

Now we consider (v).  Let $\chi\in \Xi(T)$.  Then $\chi\in \Xi_i(T)$ for some $i\in I$.  By hypothesis, there exists $\lambda_i\in Z(\Xi)_T$ such that $\mu(\Xi_i,T,\lambda_i)>0$.  Then $\langle \lambda_i,\chi_i\rangle > 0$.  Set $\lambda_\chi:=\lambda_i$.  Set $\gamma:=\sum_\chi \lambda_\chi\in Z(\Xi)_T$ (this makes sense, because $\Xi(T)$ is finite).  Now a simple calculation shows that $\mu(\Xi,T,\gamma)>0.$
\end{proof}

\begin{lem}
\label{lem:averaging}
Let $\Xi$ be a bounded $K$-quasi-state and let $H$ be a subgroup of $G$.
Then
$\Theta:=\bigcup_{h\in H} h_*\Xi$
is a bounded $K$-quasi-state, which is admissible wherever $\Xi$ is and quasi-admissible if $\Xi$ is.
Moreover, $H \subseteq C_G(\Theta)$.
\end{lem}

\begin{proof}
That $\Theta$ is a bounded $K$-quasi-state with the required admissibility properties follows from
Lemma \ref{lem:unionofstates}, setting $I = H$, and $\Xi_h = h_*\Xi$ for each $h \in H$, and from Lemma \ref{lem:propertiesofstates}(v) and (vi).
That $H \subseteq C_G(\Theta)$ is obvious from the definition of $\Theta$.
\end{proof}

The motivation for our definition of a quasi-state is that it is
precisely what is needed to capture the properties of saturated
convex polyhedral cones in $Y(K)$.  This is the content of our next results.

\begin{lem}
\label{lem:qstate1}
For $\Xi$ a quasi-admissible $K$-quasi-state of $G$ and $g \in G$, we have
$g \cdot Z(\Xi) = Z(g_* \Xi)$.
In particular, $C_G(\Xi) \subseteq N_G(Z(\Xi))$.
\end{lem}

\begin{proof}
Let $\lambda \in Y_T(K)$.
Then, by Lemma \ref{lem:propertiesofstates}(ii), we have
$\mu(g_*\Xi,gTg^{-1},g\cdot \lambda) =  \mu(\Xi,T,\lambda)$.
Therefore, $\lambda \in Z(\Xi)$ if and only if
$g\cdot \lambda \in Z(g_*\Xi)$ and the result follows.
\end{proof}

\begin{lem}
\label{lem:qstate2}
Let $C\subseteq Y(K)$ be a saturated convex polyhedral cone of finite type.
Then there exists a bounded quasi-admissible $K$-quasi-state $\Xi(C)$ such that
$C=Z(\Xi(C))$ and $N_G(C) = C_G(\Xi(C))$.
Moreover, if $K=\QQ$, then we can take $\Xi(C)$ to be integral.
\end{lem}

\begin{proof}
Fix a maximal torus $T_0$ of $G$.
Since $C$ is of finite type, the set $\{g\cdot C_{g\inverse T_0g}\mid g \in G\}$
gives a finite number of cones in $Y_{T_0}(K)$: call these cones $C_1,\ldots,C_r$.
Since $C$ is polyhedral, for each $C_i$ we can find a finite set
$D_i  \subset X_{T_0}(K)$
such that $C_i$ is the cone defined by $D_i$.
Moreover, if $K= \QQ$, then we can pick each $D_i$ to be a subset of $X_{T_0}$.

We define a quasi-state $\Xi_0$ as follows:
For each maximal torus $T$ of $G$, let $H(T)$ be the transporter of the fixed torus $T_0$ to $T$;
i.e., the set of $g \in G$ such that $gT_0g\inverse = T$.
Note that for any $g,h \in H(T)$, we have $gh\inverse \in N_G(T)$ and $h\inverse g \in N_G(T_0)$.
For each $g \in H(T)$, we have $g\inverse\cdot C_T = C_i$ for some $1\leq i \leq r$,
and we set $D_g = g_!D_i$, which gives a finite subset of $X_T(K)$.
Note that if $g,h \in H(T)$ are such that $gh\inverse \in T$, then
$D_g = D_h$, since $T$ is abelian.
So there are only finitely many different subsets $D_g$ arising in this way
(we get a number less than or equal to the order of the Weyl group of $G$).
Moreover, for any $g \in H(T)$, we see by construction that $C_T$ is the cone defined by $D_g$ in
$Y_T(K)$.
Define a quasi-state $\Xi_0$ by
$$
\Xi_0(T) := \bigcup_{g \in H(T)} D_g \qquad \textrm{ for each maximal torus } T \textrm{ of } G.
$$
Note that this is a finite set for each $T$, so $\Xi_0$ is a $K$-quasi-state, and $\Xi_0$ is integral if $K = \QQ$.
Also, since each $D_g$ defines the cone $C_T$, we have that $\Xi_0(T)$ defines the cone $C_T$ in $Y_T(K)$.

We claim that $\Xi_0$ is bounded.
To see this, let $T$ be a maximal torus of $G$, and let $g \in G$.
Then we have $H(g\inverse Tg) = g\inverse H(T)$,
so
\begin{equation}\label{eqn:technical}
(g_*\Xi_0)(T)
= g_!\left(\Xi_0(g\inverse Tg)\right)
= g_!\left(\bigcup_{h \in H(g\inverse T g)} D_h \right)
= g_!\left(\bigcup_{x \in H(T)} D_{g\inverse x} \right).
\end{equation}
Now each $D_{g\inverse x}$ has the form $(g\inverse x)_!D_i$ for some $1 \leq i \leq r$,
so $g_!D_{g\inverse x}$ has the form $x_!D_i$ for some $1 \leq i \leq r$.
Further, if $x,y \in H(T)$ are such that $y\inverse x \in T_0$, we have $x_!D_i = y_!D_i$ for all $i$.
Hence there are only finitely many possibilities for $x_!D_i$ as $x$ runs over $H(T)$ and $i$ runs over the indices $1,\ldots,r$.
Since each $D_i$ is a finite set, we can conclude that the set
$$
\bigcup_{1\leq i \leq r} \bigcup_{x \in H(T)} x_!D_i
$$
is finite.
Since \eqref{eqn:technical} shows that $(g_*\Xi_0)(T)$ is contained in this set for all $g \in G$,
we see that $\Xi_0$ is bounded, as claimed.

We claim further that $\Xi_0$ is quasi-admissible.
To see this, suppose $T$ is a maximal torus of $G$, and $\lambda \in Y_T(K)$ is such that $\mu(\Xi_0,T,\lambda) \geq 0$.
Then, since $\Xi_0(T)$ defines the cone $C_T$ in $Y_T(K)$, we have $\lambda \in C$.
Now for any $x \in P_\lambda$, we have $x\cdot \lambda \in C$, since $C$ is saturated.
Thus $x\cdot \lambda \in C_{xTx\inverse}$, which is the cone in $Y_{xTx\inverse}(K)$ defined by $\Xi_0(xTx\inverse)$.
Thus $\mu(\Xi_0,xTx\inverse,x\cdot\lambda) \geq 0$, as required.
Moreover, since $\Xi_0(T)$ defines the cone $C_T$ in each $Y_T(K)$, we have $C = Z(\Xi_0)$.

Finally, we can prove the result claimed.
We define a new quasi-state
$$
\Xi := \Xi(C) := \bigcup_{g \in N_G(C)} g_*\Xi_0.
$$
Then, by Lemma \ref{lem:averaging}, since $\Xi_0$ is a bounded quasi-admissible $K$-quasi-state, $\Xi$ is a bounded quasi-admissible $K$-quasi-state,
and $N_G(C)\subseteq C_G(\Xi)$.
Moreover, if $K = \QQ$, then $\Xi$ is integral, since $\Xi_0$ is.
Since $C = Z(\Xi_0)$, thanks to Lemma \ref{lem:qstate1} we have
$Z(g_*\Xi_0) = g\cdot Z(\Xi_0) =  g\cdot C =  C$ for each $g \in N_G(C)$,
so $C = Z(\Xi)$.
Lemma \ref{lem:qstate1} also shows that $C_G(\Xi) \subseteq N_G(C)$, so we are done.
\end{proof}

\begin{rem}
Note that the construction of the quasi-state $\Xi(C)$ associated to
the cone $C$ in Lemma \ref{lem:qstate2}
depends on the choice of the sets $D_i$ in the first paragraph of the proof,
and different choices here may give rise to different quasi-states.
However, $\Xi(C)$ does enjoy the following ``functorial'' property:
for any $g \in G$, the quasi-state $g_*(\Xi(C))$ defines
the cone $g\cdot C$ in $Y(K)$, and $N_G(g\cdot C) = C_G(g_*\Xi(C))$.
\end{rem}

\begin{cor}
\label{cor:qstatecone}
Let $\Xi$ be a bounded quasi-admissible $K$-quasi-state.
Then $Z(\Xi)$ is a saturated convex polyhedral cone
of finite type in $Y(K)$.
In addition, $C_G(\Xi) \subseteq N_G(Z(\Xi))$.
Conversely, let $C\subseteq Y(K)$ be a saturated convex polyhedral
cone of finite type.  Then there exists a bounded quasi-admissible
$K$-quasi-state $\Xi(C)$ such that $C = Z(\Xi(C))$.
Moreover, we can choose $\Xi(C)$ so that
$N_G(C) = C_G(\Xi(C))$.
\end{cor}

\begin{proof}
Let $C= Z(\Xi)$.  Since $\mu(\Xi,T,\lambda)=
\min_{\alpha\in \Xi(T)} \langle \lambda,\alpha\rangle$,
we have $\mu(\Xi,T,\lambda) \geq 0$ if and only if
$\langle \lambda,\alpha\rangle \geq 0$ for all $\alpha \in \Xi(T)$.
Hence $C_T$ is the convex polyhedral cone defined by the finite set $\Xi(T)$.
It follows easily from the boundedness of $\Xi$ that $C$ is of finite type,
and the quasi-admissibility of $\Xi$ implies that $C$ is saturated.

The remaining statements follow from
Lemma \ref{lem:qstate2}.
\end{proof}

\begin{defn}
Corollary \ref{cor:qstatecone} provides us with the key link between cones and quasi-states.
In view of this corollary, given a convex cone $C$ in $Y(K)$ and a quasi-admissible $K$-quasi-state $\Xi$,
we say that \emph{$\Xi$ defines $C$}, or \emph{$C$ is defined by $\Xi$}, if $C = Z(\Xi)$.
\end{defn}

\begin{rem}
\label{rem:needtwostates}
Suppose $C = Z(\Xi)$ is a convex cone defined by the quasi-admissible quasi-state $\Xi$.
If $C$ is not $Y(K)$-cr, one might hope that this is reflected in the values of
the numerical functions $\mu(\Xi,T,\cdot)$: for example, if there exists $\lambda \in C$
such that $0< \mu(\Xi,T,\lambda)<\infty$ for some maximal torus $T$, then we have
$\mu(\Xi,T,-\lambda)<0$, so $-\lambda \not\in C$.
The quasi-admissibility of $\Xi$ then
implies that $-(u\cdot \lambda)\not\in C$ for any $u\in R_u(P_\lambda)$, so $C$ is not $Y(K)$-cr.
However, it can happen that the numerical functions $\mu(\Xi,T,\cdot)$ are
identically zero on $C$ --- we might have $0\in \Xi(T)$ for every maximal torus $T$,
for example --- and this doesn't give us enough information to work with.
In order to get around this problem, we are forced to consider \emph{two} quasi-states:
one defining $C$, and one picking out certain points of $C$ without an opposite in $C$.
This is the reason that our results below (and those in \cite{kempf}) involve two quasi-states
$\Xi$ and $\Upsilon$.
\end{rem}




We continue by recalling an important lemma of Kempf \cite[Lem.\ 2.3]{kempf}.

\begin{lem}
\label{lem:cvxopt}
Let $E$ be a finite-dimensional real vector space with a
norm $\left\|\, \right\|$ arising from a positive definite $\RR$-valued
bilinear form.
Let $A$ and $B$ be finite subsets of $E^*$.  Define $a, b\colon E\ra \RR$ by
$a(v)=\min_{\alpha\in A}\alpha(v)$
and $b(v) = \min_{\beta\in B}\beta(v)$.
Assume that the cone $C = \{v\in E\mid a(v) \geq 0\}$ contains more than just the zero vector.
Then the following hold:
\begin{itemize}
\item[(i)] The function $v \mapsto b(v)/\left\|v\right\|$ attains a
maximum value $M$ on $C\setminus\{0\}$.
\item[(ii)] If the maximum value from (i) is finite and positive, then there is a
unique ray $R$
in $C$ such that for all $v\in C$, we have $b(v)/\left\|v\right\|=M$
if and only if $v\in R$.
\end{itemize}
Suppose further that the inner product and each function in $A$ and $B$ are integer-valued on some lattice $L$ in $E$.
Then the following hold:
\begin{itemize}
\item[(iii)] $L\cap R$ is non-empty.
\item[(iv)] $L\cap R$ consists of all positive integral multiples of the unique shortest element in $L\cap R$.
\end{itemize}
\end{lem}

\begin{rem}
Note that if the set $B$ in Lemma \ref{lem:cvxopt} is empty, we have $M=\infty$.
Parts (ii)--(iv) do not apply in this case, because $M$ is not finite.
\end{rem}

Translating the above result into our setting gives the following corollary.

\begin{cor}
\label{cor:toruscase}
Let $\Xi$ and $\Upsilon$ be real quasi-states and let
$T$ be a maximal torus of $G$.
Let $C_T = \{\lambda \in Y_T(\RR)\mid \mu(\Xi,T,\lambda) \geq 0\}$.
Then the following hold:
\begin{itemize}
\item[(a)] The function $\lambda \mapsto \mu(\Upsilon,T,\lambda)/\left\|\lambda\right\|$
has a maximum value $M(T,\Xi,\Upsilon)$ on
$C_T\setminus\{0\}$, if
this set is non-empty.
\item[(b)] If the maximum value from (i) is finite and positive, then the following hold:
\begin{itemize}
\item[(i)] There exists a unique ray $R$ in $C_T\setminus\{0\}$ such that
$\mu(\Upsilon,T,\lambda)/\left\|\lambda\right\| = M(T,\Xi,\Upsilon)$ if and only if $\lambda \in R$.
\item[(ii)] If $\Xi$ and $\Upsilon$ are rational quasi-states,
then $R\cap Y_T(\QQ)$ is a ray in $Y_T(\QQ)$.
\item[(iii)] If $\Xi$ and $\Upsilon$ are integral quasi-states, then $R\cap Y_T$ is non-empty and consists
of all positive integer multiples of the shortest element in $R\cap Y_T$.
\end{itemize}
\end{itemize}
\end{cor}

\begin{proof}
Parts (a) and (b)(i) follow from Lemma \ref{lem:cvxopt}(i) and (ii),
setting $E = Y_T(\RR)$ with the norm $\left\|\, \right\|$
we have fixed, and with $A=\Xi(T)$ and $B=\Upsilon(T)$.



For (b)(ii), first note that the norm on $Y_T(\RR)$
arises from an integer-valued form on $Y_T$ (by Definition \ref{def:norm}).
If $\Xi$ and $\Upsilon$ are $\QQ$-quasi-states, then their numerical functions take rational values on $Y_T$ and there is a
sublattice of $Y_T$ upon which they take integer values.
By parts (iii) and (iv) of Lemma \ref{lem:cvxopt}, the ray defined by $\lambda$ intersects this lattice,
and so $R\cap Y_T(\QQ)$ is non-empty, and is hence a ray in $Y_T(\QQ)$.

For (b)(iii), we can apply Lemma \ref{lem:cvxopt}(iii) and (iv) with $L=Y_T$.
\end{proof}

\begin{rem}
\label{rem:locmax}
The previous result shows that if $\Xi$ is a quasi-admissible quasi-state, $C= Z(\Xi)$,
and $\Upsilon$ is another quasi-state, then $\Upsilon$ can be used to pick out certain rays in
the subsets $C_T$ of $C$ as $T$ ranges over the maximal tori of $G$.
Roughly speaking, each such ray is the set of points in $Y_T(\RR)$ where the numerical function $\mu(\Upsilon,T,\cdot)$ attains a maximum, for some maximal torus $T$ of $G$; we call these points local maxima.
The key to Kempf's constructions in \cite{kempf},
and to our generalizations in this paper, is to impose an extra condition
on $\Upsilon$ to ensure that these local maxima patch together nicely inside all of $Y(\RR)$;
this is where the notion of admissibility becomes important.
We formalize these ideas in the following definition.
\end{rem}

\begin{defn}\label{def:localmaxima}
Let $\Xi$ and $\Upsilon$ be bounded real quasi-states, and suppose $\Xi$ is quasi-admissible.
Let $C= Z(\Xi) \subseteq Y(\RR)$.
For each maximal torus $T$ of $G$, if $C_T = \{0\}$, then set $M(T) = -\infty$.
Otherwise, let $M(T) = M(T,\Xi,\Upsilon)$ be the maximum value provided by
Corollary \ref{cor:toruscase}.
We call a point $\lambda \in C$ a \emph{local maximum of $\Upsilon$ in C}
if there exists a maximal torus $T$ of $G$ such that $\lambda \in C_T$ and
$0 < \mu(\Upsilon,T,\lambda)/\left\|\lambda\right\| = M(T) < \infty$.
\end{defn}

\begin{rem}
Note that if $C = Z(\Xi)$ and $\Upsilon$ are as in Definition \ref{def:localmaxima},
then $\Upsilon$ has a local maximum on $C$ if and only if there exists a maximal torus $T$ of $G$ and
some $\lambda \in C_T$ such that $0<\mu(\Upsilon,T,\lambda)<\infty$.
\end{rem}

We now present a generalization of Kempf's central result \cite[Thm.\ 2.2]{kempf}.
Kempf's proof goes through almost word for word if one replaces the bounded admissible states $\Xi$ and $\Upsilon$ with bounded admissible quasi-states.
The essential difference between our result and Kempf's is that in Theorem~\ref{thm:kempf2.2}(b) we
just require that the quasi-state $\Xi$ is quasi-admissible, and that the quasi-state $\Upsilon$
is admissible at its local maxima in $C = Z(\Xi)$, cf. Definition \ref{def:localmaxima}.




\begin{thm}
\label{thm:kempf2.2}
Let $\Xi$ and $\Upsilon$ be bounded real quasi-states, and suppose $\Xi$ is quasi-admissible.
Let $C= Z(\Xi) \subseteq Y(\RR)$.
For each maximal torus $T$ of $G$ such that $C_T \ne \{0\}$ let $M(T) = M(T,\Xi,\Upsilon)$ be as in Definition \ref{def:localmaxima}.
Then the following hold:
\begin{itemize}
\item[(a)] The set $\{M(T) \mid T \textrm{ is a maximal torus of } G,\ M(T)<\infty\}$
is finite, and hence has a maximum value $M$.
\item[(b)] Suppose $M$ from (a) is positive, so that $\Upsilon$ has local maxima in $C$.
If $\Upsilon$ is admissible at its local maxima in $C$,
then the set $\Lambda := \Lambda(\Xi,\Upsilon)$ of
$\lambda \in C$ such that $\left\|\lambda\right\|=1$ and
$\mu(\Upsilon,T,\lambda) = M$ for some maximal torus $T$ of $G$
has the following properties:
\begin{itemize}
\item[(i)] $\Lambda$ is non-empty;
\item[(ii)] there is a parabolic subgroup $P=P(\Xi,\Upsilon)$ of $G$ such that
$P= P_\lambda$ for any $\lambda \in \Lambda$;
\item[(iii)] $R_u(P)$ acts simply transitively on $\Lambda$;
\item[(iv)] for each maximal torus $T'$ of $P$
there is a unique $\lambda \in \Lambda \cap Y_{T'}(\RR)$;
\item[(v)] if $\Xi$ and $\Upsilon$ are rational quasi-states, then
    some positive
multiple of each $\lambda \in \Lambda$ lies in $Y$.
\end{itemize}
\end{itemize}
\end{thm}

\begin{proof}
We follow the idea of Kempf's proof \cite[Thms.\ 2.2 and 3.4]{kempf} closely.
Fix a maximal torus $T_0$ of $G$, and let $T$ be any other maximal torus.
Then $T = g\inverse T_0g$ for some $g \in G$.
Since Lemma \ref{lem:propertiesofstates}(ii) implies that
$g\cdot C_T = \{\lambda \in Y_{T_0}(\RR) \mid \mu(g_*\Xi,T_0,\lambda) \geq 0\} = Z(g_*\Xi)_{T_0}$
and that
$\mu(\Upsilon,T,\lambda)= \mu(g_* \Upsilon, gTg\inverse, g\cdot \lambda)$
for any $\lambda\in Y_T(\RR)$, the maximum values
of the function $\lambda \mapsto \mu(\Upsilon,T,\lambda)/\left\|\lambda\right\|$ on
$C_T$ and
the function $\lambda \mapsto \mu(g_*\Upsilon,T_0,\lambda)/\left\|\lambda\right\|$ on
$g\cdot C_T$ are equal, and this maximum value is $M(T)$.
Since $\Xi$ and $\Upsilon$ are bounded, there are only finitely many
possibilities for $g_*\Xi$ and $g_*\Upsilon$,
and so there is only a finite number of values $M(T)$ arising.
This proves (a).

Now assume that $M$ is positive, so that $\Upsilon$ has local maxima in $C$,
and suppose that $\Upsilon$ is admissible at its local maxima in $C$.
Choose a local maximum $\lambda_1 \in Y(\RR)\setminus\{0\}$ such that
$\lambda_1 \in C_{T_1}$ for some maximal torus $T_1$ and
$\mu(\Upsilon,T_1,\lambda_1)/\left\|\lambda_1\right\|=M$.
Multiplying $\lambda_1$ by a
positive scalar, we can ensure that $\left\|\lambda_1\right\| =1$.
This proves part (b)(i).

Now suppose that $\lambda_2$ is any other
element of $\Lambda$, and let $T_2$ be a maximal torus for which
$\lambda_2 \in C_{T_2}$, $\left\|\lambda_2\right\| =1$ and
$\mu(\Upsilon,T_2,\lambda_2) = M$.
We can choose a maximal torus
$T\subseteq P_{\lambda_1}\cap P_{\lambda_2}$.
There exists $x_1\in P_{\lambda_1}$ such that
$x_1T_1x_1\inverse = T$, and hence $x_1\cdot\lambda_1 \in Y_T(\RR)$.
Likewise there exists
$x_2\in P_{\lambda_2}$ such that
$x_2T_2x_2\inverse = T$, and hence $x_2\cdot \lambda_2\in Y_T( \RR)$.
Note that we have $0 \leq \mu(\Xi,T,x_i\cdot \lambda_i) < \infty$ for $i=1,2$,
by the quasi-admissibility of $\Xi$,
so $x_i \cdot \lambda_i \in C_T$ for $i=1,2$.
Moreover, we have
$\mu(\Upsilon,T,x_i\cdot\lambda_i)/\left\|x_i\cdot\lambda_i\right\|
= \mu(\Upsilon,T_i,\lambda_i)/\left\|\lambda_i\right\| = M$ for
$i=1,2$, by Eqn.\ \eqref{eqn:ipinvce}, the admissibility of
$\Upsilon$ at local maxima of $C$, and the $G$-invariance of the norm.
But $M$ is the maximum possible finite value of
$\mu(\Upsilon,T',\lambda)/\left\|\lambda\right\|$
on $C_{T'}$ as $T'$ ranges over all maximal tori of $G$,
hence is the maximum value on $C_T$.
By the uniqueness statement in Corollary \ref{cor:toruscase}(b)(i), we conclude
that, as
$\left\|x_1\cdot\lambda_1\right\| = \left\|x_2\cdot\lambda_2\right\| =1$,
we have $x_1\cdot\lambda_1 = x_2\cdot\lambda_2$.
Thus $P_{\lambda_1}= P_{x_1\cdot\lambda_1}
= P_{x_2\cdot\lambda_2}= P_{\lambda_2}$.
This proves parts (b)(ii) and (iv).

The arguments of the previous paragraph show
that $P$ acts transitively on $\Lambda$.
Given $\lambda_1 , \lambda_2 \in \Lambda$ and $x \in P$ such that
$\lambda_2 = x\cdot \lambda_1$,
we can write $x = ul$ with $u \in R_u(P)$ and
$l \in L_{\lambda_1} = C_G(\lambda_1)$.
Then $\lambda_2 = u\cdot \lambda_1$, hence $R_u(P)$ acts transitively
on $\Lambda$.
Now if $u\cdot\lambda_1 = u'\cdot\lambda_1$ for $u, u' \in R_u(P)$, then
$u\inverse u' \in L_{\lambda_1}\cap R_u(P) = \{1\}$,
hence $u=u'$. This proves part (b)(iii).

For the final statement (b)(v), pick some $\lambda \in \Lambda$ and some maximal torus $T$ such that $\lambda \in Y_T(\RR)$.
Then by Corollary \ref{cor:toruscase}(b)(ii), the ray of all positive multiples of $\lambda$ intersects $Y_T(\QQ)$ in a ray.
Any element of $Y_T(\QQ)$ can be scaled by a positive integer to give an element of $Y_T$.
\end{proof}

\begin{defn}
\label{def:optclass}
We call
$\Lambda(\Xi,\Upsilon) \subseteq Y( \RR)$ from Theorem \ref{thm:kempf2.2}(b)
\emph{the class of optimal cocharacters afforded by the pair
of $\RR$-quasi-states $(\Xi,\Upsilon)$}.
Similarly, we call the parabolic subgroup $P(\Xi,\Upsilon)$ of $G$ \emph{the optimal parabolic subgroup
afforded by the pair $(\Xi,\Upsilon)$}.
\end{defn}

\begin{rem}
\label{rem:funct}
Let $\Xi$ and $\Upsilon$ be bounded real quasi-states
as in Theorem \ref{thm:kempf2.2}(b) and let
$P(\Xi,\Upsilon)$ be the optimal parabolic subgroup of $G$
afforded by the pair $(\Xi,\Upsilon)$ from Definition \ref{def:optclass}.
It is clear that the map
$(\Xi,\Upsilon)\mapsto P(\Xi,\Upsilon)$ is functorial in the following sense:
for any $g\in G$,
$g\cdot\Lambda(\Xi,\Upsilon) = \Lambda(g_*\Xi,g_*\Upsilon)$; hence
$gP(\Xi,\Upsilon)g\inverse = P(g_*\Xi,g_*\Upsilon)$. 
In particular, if $g \in C_G(\Xi)\cap C_G(\Upsilon)$,
then $g$ stabilizes the optimal class $\Lambda(\Xi,\Upsilon)$ and normalizes the
parabolic subgroup $P(\Xi,\Upsilon)$; hence $C_G(\Xi)\cap C_G(\Upsilon) \subseteq P(\Xi,\Upsilon)$.
\end{rem}

\section{Quasi-states and $G$-centres}
\label{sec:buildings}

In this section, we translate our results into the language of 
spherical buildings.
The principal result of the paper is
Theorem \ref{thm:partialTCC} which
gives a complete characterization of the existence of a $G$-centre
of a  convex polyhedral
subset of finite type $\Sigma$ in $\Delta(K)$
in terms of the existence of a bounded integral quasi-state which admits
local maxima on $\zeta^\inverse(\Sigma) \cup \{0\}$.

Recall the notation from Section \ref{sub:buildings}.
The key tool is the link between convex subsets of $\Delta(K)$ and
quasi-admissible $K$-quasi-states, which we briefly discuss now.
We first consider the special case of quasi-states which are admissible.

\begin{defn}
\label{def:mu}

Let $\Upsilon$ be a bounded admissible $K$-quasi-state of $G$. 
By Lemma \ref{lem:propertiesofstates}(iii),
for any $\lambda \in Y(K)$ the value of $\mu(\Upsilon,T,\lambda)$ is independent of the choice of
maximal torus $T$ with $\lambda \in Y_T(K)$.
Hence we can define a numerical function $\mu(\Upsilon,\cdot)$ on all of $Y(K)$
without ambiguity by setting
$$
\mu(\Upsilon,\lambda) := \mu(\Upsilon,T,\lambda),
$$
where $T$ is any maximal torus of $G$ such that $\lambda \in Y_T(K)$.
Moreover, since $\mu(\Upsilon,\lambda) = \mu(\Upsilon, u\cdot\lambda)$ for any $u \in R_u(P_\lambda)$,
this function descends to give a $K$-valued function on
$V(K)$, which we also denoted by $\mu(\Upsilon,\cdot)$.
Finally, we can also restrict to get a $K$-valued function on
$\Delta(K)$.

\end{defn}

Now let $\Xi$ be a bounded quasi-admissible $K$-quasi-state of $G$.
If $\Xi$ is not admissible, 
the numerical functions $\mu(\Xi, T, \cdot)$ on $Y(K)$ do not descend to
give a well-defined function on $\Delta(K)$ as in Definition \ref{def:mu},
since the value of $\mu(\Xi,T,\lambda)$ may depend on the choice of $T$ with $\lambda \in Y_T(K)$.
However, we can still form $Z(\Xi)$, which is a saturated convex polyhedral cone
of finite type in $Y(K)$, by Corollary \ref{cor:qstatecone}.
So $\zeta(Z(\Xi)\setminus\{0\})$ is a convex polyhedral set of
finite type in $\Delta(K)$.
This is analogous to considering the Zariski topology on projective varieties:
a homogeneous polynomial in $n+1$ variables does not give a well-defined function on projective $n$-space,
but its vanishing set \emph{is} well-defined.
Likewise, the numerical function of the quasi-admissible quasi-state $\Xi$ does not give a
well-defined function on $V(K)$ or $\Delta(K)$,
but it does make sense to speak of the set of points in $V(K)$ or $\Delta(K)$
where the numerical function is non-negative.

Our next result is one of the main results of the paper.
It provides the promised link between the quasi-state formalism
in Section \ref{sec:quasistates} and Conjecture \ref{conj:tcc},
and shows how to use this link to find centres for convex subsets of $\Delta(K)$.

\begin{thm}
\label{thm:bldgversion}
Let $\Sigma$ be a convex polyhedral set of finite type in
$\Delta(K)$ and let $C = \zeta\inverse(\Sigma) \cup \{0\}$.
Suppose that $\Upsilon$ is a bounded $K$-quasi-state of $G$ such that
$\Upsilon$ has local maxima on $C$ and $\Upsilon$ is admissible at these local maxima.
Then there exists a bounded quasi-admissible $K$-quasi-state $\Xi$ defining $C$ with $N_G(C) = C_G(\Xi)$.
Moreover, for \emph{any} such $K$-quasi-state $\Xi$ we have:
\begin{itemize}
\item[(i)] $\zeta(\Lambda(\Xi,\Upsilon))$ is a singleton set $\{c\}$, where $\Lambda(\Xi,\Upsilon)$
is the class of optimal cocharacters afforded by the pair $(\Xi,\Upsilon)$;
\item[(ii)] $c$ from part (i) is a $C_G(\Upsilon)$-centre of $\Sigma$.
\end{itemize}
\end{thm}

\begin{proof}
The set $C = \zeta\inverse(\Sigma) \cup \{0\}$ is a saturated convex polyhedral cone of finite type in $Y(K)$, thanks to Lemma \ref{lem:Ycones}.
So, by Corollary \ref{cor:qstatecone}, there is a quasi-admissible
bounded $K$-quasi-state $\Xi$ such that $C= Z(\Xi)$, and we can choose
$\Xi$ in such a way that $C_G(\Xi) = N_G(C)$, which proves the first assertion of the theorem.

Now suppose $\Xi$ is any bounded quasi-admissible $K$-quasi-state defining $C$ with $N_G(C) = C_G(\Xi)$.
Since $\Upsilon$ has local maxima on $C$, and $\Upsilon$ is admissible at these local maxima,
the hypotheses of Theorem \ref{thm:kempf2.2}(b) hold,
so we can define the optimal class $\Lambda(\Xi,\Upsilon)$.
If $K = \RR$, then $\zeta(\Lambda(\Xi,\Upsilon))$
is a singleton set $\{c\}$, by Theorem \ref{thm:kempf2.2}(b)(iii), which gives (i).
Now Remark \ref{rem:funct} implies that $c$
is fixed by $C_G(\Xi) \cap C_G(\Upsilon)$.
Since $C_G(\Xi) = N_G(C) = N_G(\Sigma)$,
part (ii) follows.

In the case $K=\QQ$, we have to be a little bit more careful.
We first move into $Y(\RR)$ by looking at the cone $Z(\Xi) \subseteq Y(\RR)$
(this is just the closure in $Y(\RR)$ of the corresponding cone in $Y(\QQ)$).
Now, by Theorem \ref{thm:kempf2.2}(b)(v),
since $\Xi$ and $\Upsilon$ are $\QQ$-quasi-states,
we have $\{c\} = \zeta(\Lambda(\Xi,\Upsilon)) \subseteq \Delta(\QQ)$
so $c \in \Sigma$.
\end{proof}

\begin{cor}
\label{cor:admissiblecase}
Suppose that $\Sigma$ is a convex polyhedral
subset of finite type in $\Delta(K)$, and let $C = \zeta\inverse(\Sigma) \cup \{0\}$.
Suppose there is a bounded admissible $K$-quasi-state $\Upsilon$ of $G$ such that
$\mu(\Upsilon,\cdot)$ attains a finite positive value on $C$.
Then $\Sigma$ has a $C_G(\Upsilon)$-centre.
If further $N_G(\Sigma) \subseteq C_G(\Upsilon)$, then $\Sigma$ has a $G$-centre.
\end{cor}

\begin{proof}
Since $\Upsilon$ attains a finite positive value on $C$,
it has local maxima on $C$.
Since $\Upsilon$ is admissible, it is certainly admissible at local maxima in $C$,
so we can apply Theorem \ref{thm:bldgversion}.
The second assertion follows immediately.
\end{proof}

\begin{rem}
\label{rem:oppos}
Let $\Sigma$ be a convex polyhedral
subset of finite type in $\Delta(K)$.
Note that in Theorem \ref{thm:bldgversion}, Theorem \ref{thm:partialTCC} and
Corollary \ref{cor:admissiblecase}, we do not assume that $\Sigma$ is not $\Delta(K)$-cr,
and yet we still find a centre.
However, the assumptions on the existence of $\Upsilon$
do restrict the possibilities for $\Sigma$ in practice.

For example, in Corollary \ref{cor:admissiblecase}, we have that
$\mu(\Upsilon,\lambda) >0$ for some $\lambda \in C = \zeta\inverse(\Sigma)\cup \{0\}$.
This implies that $\lambda \in Z(\Upsilon) \cap C$, but
$-\lambda \not\in Z(\Upsilon) \cap C$ (cf.\ Remark \ref{rem:needtwostates}).
Thus the image of $Z(\Upsilon)\setminus\{0\} \cap C$ in $\Delta(K)$, which is
$\zeta(Z(\Upsilon)\setminus\{0\})\cap \Sigma$, is a subset of $\Delta(K)$ which is not $\Delta(K)$-cr,
and our centre actually lies in this set.
\end{rem}

Our final theorem of this section is the central result of the paper.
It shows that not only do our methods involving quasi-states
suffice to guarantee the existence of a $G$-centre
of a convex polyhedral subset of $\Delta(K)$,
but that the existence of a suitable
quasi-state is actually necessary.

\begin{thm}
\label{thm:partialTCC}
Let $\Sigma$ be a convex polyhedral
subset of finite type in $\Delta(K)$, and let $C = \zeta\inverse(\Sigma) \cup \{0\}$.
Then $\Sigma$ has a $G$-centre if and only if there is a bounded integral quasi-state
$\Upsilon$ such that $\Upsilon$ has local maxima on $C$, $\Upsilon$ is admissible at these local maxima,
and $N_G(\Sigma) \subseteq C_G(\Upsilon)$.
\end{thm}

\begin{proof}
Suppose $\Sigma$ has a $G$-centre $c$.
Let $\lambda \in Y(K)$ be such that $\zeta(\lambda) = c$.
Fix a maximal torus $T_0$ of $G$ such that $\lambda \in Y_{T_0}(K)$,
and let $P = P_\lambda$ be the (proper) parabolic subgroup of $G$ attached to $\lambda$.
We construct $\Upsilon$ with the desired properties directly;
the construction is similar to that employed in
Lemmas \ref{lem:qstate2} and \ref{lem:cxcvx}.

First, let $\Psi = \Psi(G,T_0)$ be the root system of $G$ with respect to $T_0$, and define
$$
\Upsilon(T_0):= \{\alpha \in \Psi \mid U_\alpha \subseteq R_u(P)\} = \Psi(R_u(P),T_0).
$$
Now, for any other maximal torus $T$ of $G$ such that $T \subset P$,
choose $g \in P$ such that $gT_0g\inverse = T$, and set $\Upsilon(T) = g_!\Upsilon(T_0)$.
Finally, for any maximal torus $T$ of $G$ which is not contained in $P$,
set $\Upsilon(T) = \varnothing$.

We first claim that $\Upsilon$ is well-defined.
This amounts to showing that the construction of $\Upsilon(T)$
for $T \subset P$ is independent of the choice of $g \in P$ with $gT_0g\inverse = T$.
To see this, suppose that $h \in P$ is such that $hT_0h\inverse = T$.
Then $h\inverse g \in N_P(T_0)$, and $N_P(T_0)$ stabilizes the set of roots $\Psi(R_u(P),T_0) = \Upsilon(T_0)$,
so we have $g_!\Upsilon(T_0) = h_!\Upsilon(T_0)$, as required.

Now, it is clear that $\Upsilon$ is an integral quasi-state.
The fact that $\Upsilon$ is bounded follows from
arguments similar to those in the proof of Lemma \ref{lem:qstate2}.

To show that $\Upsilon$ is admissible at local maxima, we first look at the stabilizer of $\Upsilon$.
Let $T$ be any maximal torus of $P$, and find $g \in P$ such that $gT_0g\inverse = T$;
then by construction $\Upsilon(T) = g_!\Upsilon(T_0)$.
Now for any $x \in P$, we have $x\inverse g \in P$ and
$(x\inverse g)T_0(x\inverse g)\inverse = x\inverse T x$,
so $\Upsilon(x\inverse Tx) = (x\inverse g)_!\Upsilon(T_0)$.
Thus we have
$$
(x_*\Upsilon)(T) = x_!\Upsilon(x\inverse Tx) = x_!((x\inverse g)_!\Upsilon(T_0))
= g_!\Upsilon(T_0)= \Upsilon(T),
$$
which shows that $(x_*\Upsilon)(T) = \Upsilon(T)$ for all $x \in P$.
On the other hand, if $T$ is a maximal torus of $G$ not contained in $P$,
then $\Upsilon(T) = \varnothing$ and $\Upsilon(x\inverse Tx) = \varnothing$ for all $x \in P$,
so we have $(x_*\Upsilon)(T) = \Upsilon(T)$ in this case as well.
This shows that $P \subseteq C_G(\Upsilon)$.
Now suppose $x \in C_G(\Upsilon)$.
Then $\Upsilon(T_0) = (x_*\Upsilon)(T_0) = x_!\Upsilon(x\inverse T_0x)$.
This implies that $\Upsilon(x\inverse T_0x)$ is non-empty, so $x\inverse T_0x \subset P$.
Find $g \in P$ such that $gT_0g\inverse = x\inverse T_0 x$; then $xg \in N_G(T_0)$ and
$$
\Upsilon(T_0) = x_!\Upsilon(x\inverse T_0x) = x_! \Upsilon(gT_0g^{-1}) = x_!g_!\Upsilon(T_0) = (xg)_!\Upsilon(T_0).
$$
Consequently,
$xg$ is in the subgroup of $N_G(T_0)$ consisting of the elements that stabilize
$\Upsilon(T_0) = \Psi(R_u(P),T_0)$.
But $R_u(P)$ is generated by the root groups $U_\alpha$ with $\alpha \in \Psi(R_u(P),T_0)$,
so $xg \in N_G(R_u(P)) = P$.
Since $g \in P$, we have $x \in P$, and
thus $C_G(\Upsilon) \subseteq P$.
Combining these inclusions, we get $C_G(\Upsilon) = P$.

Now suppose $\nu \in Y(K)$ is such that $0 < \mu(\Upsilon,T,\nu) < \infty$ for
some maximal torus $T$ of $G$ with $\nu \in Y_T(K)$.
Then $T \subset P$, because $\mu(\Upsilon,T,\nu)$ has a finite value,
so there exists $g \in P$ such that $gT_0g\inverse = T$
and $\Upsilon(T) = g_!\Upsilon(T_0)$.
Now $\mu(\Upsilon,T,\nu) > 0$ implies that $\langle \nu, \alpha \rangle >0$
for all $\alpha \in \Upsilon(T)$, which implies that
$\langle \nu, g_!\beta \rangle = \langle g\inverse\cdot\nu,\beta \rangle >0$
for all $\beta \in \Upsilon(T_0) = \Psi(R_u(P), T_0)$.
So we have $R_u(P) \subseteq R_u(P_{g\inverse\cdot\nu})$, so $P_{g\inverse\cdot\nu} \subseteq P$.
But $g\in P$, so we conclude that $P_\nu \subseteq P$.
Therefore, for any $x \in P_\nu$, we have $x \in P$, so $(x_*\Upsilon)(T) = \Upsilon(T)$
by the previous paragraph.
Thus
$$
\mu(\Upsilon,xTx\inverse,x\cdot\nu) = \mu(x_*^{-1}\Upsilon,T,\nu) = \mu(\Upsilon,T,\nu)
$$
for all $x \in P_\nu$, by Lemma \ref{lem:propertiesofstates}(ii).
This shows that $\Upsilon$ is admissible at all points where its numerical function
takes a finite positive value.

Let $C = \zeta\inverse(\Sigma) \cup \{0\}$.
By construction, $0<\mu(\Upsilon,T_0,\lambda) <\infty$, so $\Upsilon$ has local maxima on $C$,
and by the previous paragraph $\Upsilon$ is admissible at these local maxima.
Moreover, since $N_G(\Sigma)$ fixes $c$, and the function $\zeta:Y(K)\setminus\{0\} \to \Delta(K)$ is
$G$-equivariant, we must have that $N_G(\Sigma)$ normalizes $P = P_\lambda$,
and hence $N_G(\Sigma) \subseteq P = C_G(\Upsilon)$.
This proves the forward implication of the result.

The other direction follows immediately from Corollary \ref{cor:admissiblecase}.\end{proof}

\begin{rems}
\label{rems:howtoprovetcc}
(i). Note that Theorem \ref{thm:partialTCC} says that proving Conjecture \ref{conj:tcc}
(or at least finding a $G$-centre) for a convex polyhedral subset $\Sigma$ of
finite type in $\Delta(K)$ is equivalent to finding a suitable
quasi-state $\Upsilon$ whose numerical function is sufficiently well-behaved on
$\zeta\inverse(\Sigma)$.
It also says that, in theory at least, it is enough to look at integral quasi-states.
Moreover, given $\Upsilon$, we can construct a centre explicitly --- it is the image under $\zeta$ of
the optimal class of cocharacters $\Lambda(\Xi,\Upsilon)$ in the building $\Delta(K)$.
In Section \ref{sec:gitstcc} below, we show how to find such a quasi-state $\Upsilon$ in some specific cases
arising from GIT.

(ii). The quasi-state $\Upsilon$ in the proof of Theorem \ref{thm:partialTCC}
is admissible at local maxima of $C$ in our sense,
but is not necessarily admissible on all of $Y(K)$.
This result shows why it is important to weaken Kempf's original notions in
Definition \ref{defn:numer}, cf. Remark \ref{rem:kempforus}.
Despite this difficulty,
in our applications in Section \ref{sec:gitstcc} below, we are usually
able to find quasi-states $\Upsilon$ which \emph{are} admissible. 

(iii). 
The $G$-centre provided
by the quasi-state $\Upsilon$ may not be the same as the original $G$-centre
given in the statement of Theorem \ref{thm:partialTCC}.
For a simple example of this, consider a proper parabolic subgroup $P$ of $G$,
and let $\Sigma = \Delta(K)^P$ be the subcomplex consisting of the simplices
in $\Delta(K)$ that are contained in $\sigma_P$.
Then it is easy to see that $N_G(\Sigma) = P$.
Now, given any $\lambda \in Y(K)$ such that $P_\lambda = P$,
we have that $\zeta(\lambda)$
is fixed by $N_G(\Sigma)$; hence $\zeta(\lambda)$
is a $G$-centre of $\Sigma$, and $\Sigma$ has infinitely many $G$-centres
in general.
However, the quasi-state $\Upsilon$ constructed in the proof of Theorem \ref{thm:partialTCC}
depends only on $P$, and so picks out just one of these $G$-centres, whatever our initial choice of $\lambda$ was.
\end{rems}

\section{GIT and the Centre Conjecture}
\label{sec:gitstcc}


We now recall how Kempf's results
on GIT and optimal parabolic subgroups follow from his result
\cite[Thm.~2.2]{kempf} on states, and we recast his proof
in the language of buildings and centres.
We use Theorem \ref{thm:kempf2.2}
--- our extension of \cite[Thm.~2.2]{kempf} ---
to strengthen Kempf's results.
This allows us to deal with a special case of the
Centre Conjecture in which the subset $\Sigma$ of $\Delta(K)$
comes from a set of destabilizing cocharacters for some $G$-action.
We then illustrate these ideas by proving some further
cases of the Centre Conjecture (Theorems \ref{thm:apt} and \ref{thm:tcc});
these last two results provide applications of the GIT-methods in this
paper to situations which have no apparent connection with GIT.
This is rather striking, and supports our view that these methods
provide valuable insight
into Conjecture \ref{conj:tcc}.


Recall the notation and terminology set up in Section \ref{sec:instab}.
In particular, fix an affine $G$-variety $A$, a subset $U$ of $A$ and a closed $G$-stable
subvariety $S$ of $A$.
We introduce one further piece of notation to help elucidate the connection between
the approach of \cite{kempf}, \cite{He}, \cite{GIT} and our building-theoretic approach.

\begin{defn}\label{defn:D}
Let $A$, $U$ and $S$ be as above.
We define
$$
D_{A,U}(\QQ) := \{a\lambda \mid a\in \QQ_{\geq0}, \lambda \in |A,U|\} \subseteq Y(\QQ),
$$
and we define $D_{A,U}(\RR)$ by letting $D_{A,U}(\RR)_T$ be the closure of $D_{A,U}(\QQ)_T$ in $Y_T(\RR)$.
In both cases, we let
$E_{A,U}(K) := \zeta(D_{A,U}(K)\setminus\{0\}) \subseteq \Delta(K)$.
\end{defn}

We now show that $D_{A,U}(K)$ is a convex polyhedral cone in $Y(K)$,
by associating a bounded admissible quasi-state to $|A,U|$.
The ideas in this lemma and the next follow closely those in \cite[Sec.\ 3]{kempf}
and \cite[Sec.\ 4]{GIT}, but we reproduce many of the details for the
convenience of the reader.

\begin{lem}\label{lem:ThetaAU}
There exists a bounded admissible integral quasi-state $\Theta = \Theta_{A,U}$
such that $Z(\Theta) = D_{A,U}(K)$.
In particular, $D_{A,U}(K)$ is a convex polyhedral cone of finite type in $Y(K)$.
Moreover, $D_{A,U}(K)\cap Y = |A,U|$ and $N_G(D_{A,U}(K)) = N_G(|A,U|) \subseteq N_G(|A,U|_S)$.
\end{lem}

\begin{proof}
We begin by setting up some notation, following
ideas in \cite{He} and \cite[Sec.\ 4]{GIT}.\
By \cite[Lem.\ 1.1(a)]{kempf}, we can embed $A$ $G$-equivariantly
into a finite-dimensional rational $G$-module $V$.
Now for each $x \in U$ we define an integral quasi-state $\Theta_{V,x}$ as follows:
for each maximal torus $T$ of $G$, let $\Theta_{V,x}(T)$ be the set of
weights $\chi$ of $T$ on $V$ such that the projection of $x$ on the
weight space $V_\chi$ is non-zero (cf. \cite[Lem.\ 3.2]{kempf}).
It is standard that for $\lambda \in Y_T$, $\lim_{a\to 0}\lambda(a)\cdot x$
exists if and only if $\langle \lambda,\chi \rangle \geq 0$ for all
$\chi \in \Theta_{V,x}(T)$.
By \cite[Lem.\ 3.2]{kempf}, each $\Theta_{V,x}$ is a
bounded admissible integral quasi-state. 
Now we define
\begin{equation}
\label{eqn:theta}
\Theta := \Theta_{A,U} := \bigcup_{x\in U} \Theta_{V,x}.
\end{equation}
Since for each maximal torus $T$ of $G$, the set of all weights
of $T$ on $V$ is finite, Lemma \ref{lem:unionofstates}(ii) implies that $\Theta$ is still a bounded quasi-state.
Moreover, since each $\Theta_{V,x}$ is admissible, so is $\Theta$, by Lemma \ref{lem:unionofstates}(iii).
Now, for any $\lambda \in Y$,
we have $\mu(\Theta,\lambda) \geq 0$
if and only if $\mu(\Theta_{V,x},\lambda) \geq 0$ for all $x\in U$
if and only if $\lim_{a\to 0} \lambda(a)\cdot x$ exists for all $x\in U$.
This shows that $Z(\Theta) \cap Y = |A,U|$.
Note also that if $\lambda \in Y$ and $q\lambda \in |A,U|$ for some $q\in \QQ$, then $\lambda \in |A,U|$,
which shows that $D_{A,U}(\QQ) \cap Y = |A,U|$.
It follows straight away that we also have $D_{A,U}(\RR) \cap Y = |A,U|$.

To see that $D_{A,U}(K) = Z(\Theta)$, first suppose that $K=\QQ$.
If $\lambda\in Y(\QQ)$, then $n\lambda\in Y$ for some $n\in \NN$,
and we have $\lambda\in Z(\Theta)$ (resp.\ $\lambda\in D_{A,U}(\QQ)$) if and only if $n\lambda\in Z(\Theta)$
(resp.\ $n\lambda\in D_{A,U}(\QQ)$).
But $D_{A,U}(\QQ)\cap Y= |A,U| = Z(\Theta) \cap Y$, so we are done in case $K = \QQ$.
The result for $K=\RR$ now follows from the definition of $D_{A,U}(\RR)_T$ as the closure of $D_{A,U}(\QQ)_T$
for each maximal torus $T$ of $G$, and from the basic properties of cones laid out in Section \ref{sec:cvxcones}.
The final statement that $N_G(D_{A,U}(K)) = N_G(|A,U|) \subseteq N_G(|A,U|_S)$ is now immediate.
\end{proof}

\begin{rem}\label{rem:movethetamoveU}
Let $\Theta_{A,U}$ be as in Lemma \ref{lem:ThetaAU}, and suppose $g \in G$.
Then, for any maximal torus $T$ of $G$,
$\chi \in \Theta_{A,U}(T)$ if and only if $g_!\chi \in \Theta_{A,g\cdot U}(gTg\inverse)$
if and only if $\chi \in (g_*^{-1}\Theta_{A,g\cdot U})(T)$.
This shows that $g_*\Theta_{A,U} = \Theta_{A,g\cdot U}$.
In particular, $N_G(U) \subseteq C_G(\Theta_{A,U})$.
\end{rem}

We now show how one can use the quasi-state in Lemma \ref{lem:ThetaAU}
to prove another special case of Conjecture \ref{conj:tcc}.
As we remark below, there are other ways to approach Theorem \ref{thm:apt}, but our proof serves as a first
illustration of how methods from GIT can be applied to situations which apparently do not relate to this set-up.

\begin{thm}
\label{thm:apt}
Suppose $\Sigma$ is a convex polyhedral
subset of finite type in $\Delta(\QQ)$ which
is contained within a single apartment of $\Delta(\QQ)$.
If $\Sigma$ is not $\Delta(\QQ)$-completely reducible,
then $\Sigma$ has a $G$-centre.
\end{thm}

\begin{proof}
Let $C = \zeta^{-1}(\Sigma) \cup \{0\}$.
Then $C$ is a saturated convex polyhedral cone of finite type in $Y(\QQ)$.
Let $T$ be a maximal torus of $G$ such that $\Sigma$ is contained in the apartment corresponding to $T$.
Then $C_T$ is a convex polyhedral cone in $Y_T(\QQ)$, and $\zeta(C_T\setminus\{0\}) = \Sigma$.
Since $C_T \subseteq Y_T(\QQ)$, we can find a subset
$\{\alpha_1, \ldots, \alpha_r\} \subset X_T$ defining $C_T$;
i.e., $C_T = \{\lambda \in Y_T(K) \mid \langle \lambda,\alpha_i \rangle \geq 0 \textrm{ for all } i\}$.
Suppose $\Sigma$ is not $\Delta(\QQ)$-cr; then there exists $y\in \Sigma$ such that $y$ has
no opposite in $\Sigma$, so
there exists $\lambda\in C_T$ corresponding to $y$ such that $-\lambda\not\in C_T$.
It follows that $\langle \lambda,\alpha_i\rangle> 0$ for some $i$.
To ease notation, let $\beta = \alpha_i$.

Now let $V$ be a finite-dimensional
representation of $G$ such that the
weight space $V_\beta$ with respect to $T$ is non-zero.
Let $U$ be the set of vectors $x \in V$ such that $\mu(\Theta_{V,x},\cdot)$ is non-negative on $C$ and
takes a finite positive value somewhere on $C$, where $\Theta_{V,x}$ is the
admissible quasi-state defined in the proof of Lemma \ref{lem:ThetaAU} (taking $A$ to be $V$).
We have $\langle \nu,\beta\rangle \geq 0$ for all $\nu \in C_T$, and
$\langle \lambda,\beta \rangle >0$, so for any $0 \neq x \in V_\beta$, we have $x \in U$.
Thus $U$ is a non-empty subset of $V$, and $U \neq \{0\}$.

We claim that $N_G(\Sigma) \subseteq N_G(U)$.
To see this, let $g \in N_G(\Sigma)$ and $x \in U$.
Then for all $\nu \in C$, we have $g \inverse \cdot \nu \in C$ and
thus $\mu(\Theta_{V,x},g\inverse \cdot \nu) \geq 0$, so
$$
\mu(\Theta_{V,g\cdot x}, \nu) = \mu(g_*^{-1}\Theta_{V,g\cdot x},g\inverse \cdot\nu) = \mu(\Theta_{V,x}, g\inverse \cdot \nu) \geq 0,
$$
where the first equality follows from Lemma \ref{lem:propertiesofstates}(ii) and the second from Remark \ref{rem:movethetamoveU}.
Moreover, there exists $\nu \in C$
for which these values are all positive,
and this shows that $g\cdot x \in U$, as required.

Let $\Theta_{V,U}$ be the admissible quasi-state given by Lemma \ref{lem:ThetaAU}.
Then $\Theta_{V,U}$ is the union of the quasi-states $\Theta_{V,x}$ as $x$ runs over $U$,
so we can find $\gamma \in C_T$ such that $0<\mu(\Theta_{V,U},\gamma)<\infty$, by Lemma \ref{lem:unionofstates}(v).
Moreover, $N_G(\Sigma) \subseteq N_G(U) \subseteq C_G(\Theta_{V,U})$, by the previous paragraph and Remark \ref{rem:movethetamoveU}.
We have now verified all the hypotheses necessary to apply Corollary \ref{cor:admissiblecase},
which finishes the proof.
\end{proof}

\begin{rem}
If one is working over $\RR$ instead of $\QQ$, so that it makes sense to ask whether a subset is contractible or not,
then it is known that a closed convex contractible subset of a sphere contains a centre,
and this centre is fixed by all the isometries of the sphere that stabilize the subset
(this follows for example from \cite[Lem.\ 1]{white}).
Now suppose $\Sigma$ is a closed convex contractible subset of a single apartment $\Delta_T(\RR)$
of $\Delta(K)$.
Then for any other apartment $\Delta_{T'}(K)$ of $\Delta(\RR)$ containing $\Sigma$, there exists
an isomorphism $\Delta_{T'}(\RR) \to \Delta_T(\RR)$ fixing $\Sigma$ pointwise, by the
building axioms.
Thus any $\Delta(\RR)$-automorphism stabilizing $\Sigma$ actually stabilizes $\Delta_T(\RR)$,
modulo an automorphism which fixes $\Sigma$ pointwise.
Now $\Delta_T(\RR)$ is a sphere, so the result follows.

If $\Sigma\subseteq \Delta(\RR)$ is a convex polyhedral subset of finite type which is contained in a single apartment $\Delta_T(\RR)$,
then it is easily seen that $\Sigma$ is closed.
Thus if $\Sigma$ is not $\Delta(\RR)$-completely reducible, then the argument of the previous paragraph shows that $\Sigma$ has an $\Aut(\Delta(\RR))$-centre $c$.
Hence Theorem \ref{thm:apt} also holds when $\QQ$ is replaced by $\RR$.
The above argument argument does not, however, tell us that if $\Sigma$ is defined by a $\QQ$-quasi-state, then $c$ belongs to $\Delta(\QQ)$.
Our proof of Theorem \ref{thm:apt} is therefore of independent interest.
\end{rem}

\begin{rem}\label{rem:fishcakes}
Note that for $\Theta_{A,U}$ as in Lemma \ref{lem:ThetaAU} above,
we have $\mu(\Theta_{A,U},\lambda)>0$ if and only if $\lim_{a\to 0} \lambda(a)\cdot x = 0$
for all $x \in U$.
In Theorem \ref{thm:apt} above, $\mu(\Theta_{A,U},\cdot)$ attains a positive value on the cone $C$.
However, we may find ourselves in a situation where $\mu(\Theta_{A,U},\lambda) = 0$ for all $\lambda \in D_{A,U}(K)$.
This happens, for example, if $U = \{x\}$ is a singleton and the closure of the $G$-orbit
$G\cdot x$ does not contain $0$.

Since our methods rely on optimizing over quasi-states whose numerical functions attain strictly positive values,
we have to introduce further quasi-states to the analysis.
In particular, we have to consider the quasi-state $\Upsilon$ in Proposition \ref{prop:kempfstate} below.
See also Remark \ref{rem:needtwostates}.
\end{rem}

\begin{prop}
\label{prop:kempfstate}
Suppose $U$ is properly uniformly $S$-unstable.
Then there exist bounded admissible integral quasi-states $\Xi=\Xi_{A,U}$
and $\Upsilon=\Upsilon_{A,U,S}$ such that:
\begin{itemize}
\item[(i)] $D_{A,U}(K) = Z(\Xi)$ and $N_G(D_{A,U}(K)) = C_G(\Xi)$.
\item[(ii)] $|A,U|_S =
\{\lambda\in |A,U| \mid \mu(\Upsilon,\lambda) > 0\}$.
\item[(iii)] $N_G(D_{A,U}(K)) \subseteq C_G(\Upsilon)$.
\end{itemize}



\end{prop}

\begin{proof}
Let $H := N_G(D_{A,U}(K))$, and note that $H = N_G(|A,U|)$, by Lemma \ref{lem:ThetaAU}.
Define
$$
\Xi := \Xi_{A,U} := \bigcup_{h\in H} h_*\Theta,
$$
where $\Theta$ is the integral quasi-state given in Lemma \ref{lem:ThetaAU}.
Then, by Lemma \ref{lem:averaging},
$\Xi$ is a bounded admissible integral quasi-state and $H \subseteq C_G(\Xi)$.
Moreover, $Z(\Theta) = D_{A,U}(K)$, so by Lemma \ref{lem:ThetaAU},
for every $h\in H$ we have
$$
Z(h_*\Theta) = h\cdot Z(\Theta) = h\cdot D_{A,U}(K) = D_{A,U}(K) = Z(\Theta),
$$
by Lemma \ref{lem:qstate1}.
So $Z(\Xi) = D_{A,U}(K)$ and $C_G(\Xi) \subseteq H$.
This completes the proof of part (i).

For (ii) and (iii), we find a $G$-equivariant morphism $f:A \to W$, where $W$ is
a finite-dimensional rational $G$-module and $f\inverse(\{0\}) = S$ (scheme-theoretic preimage),
as in \cite[Lem.\ 1.1(b)]{kempf}.
We then let $\Upsilon_0 = \Theta_{W,f(U)}$,
in the notation of
\eqref{eqn:theta}.
Now it is easy to see that $|A,U|_S = |A,U| \cap |W,f(U)|_{0}$.
Moreover,
$|W,f(U)|_{0} = \{\lambda\in Y\mid \mu(\Theta_{W,f(U)},\lambda)>0 \}$.
Now, if we define
$$
\Upsilon:= \Upsilon_{A,U,S} := \bigcup_{h\in H} h_*\Upsilon_0,
$$
part (iii) follow from Lemma \ref{lem:averaging}. 

If $\lambda\in |A,U|$, then $\lambda\in |A,U|_S$ if and only if $\mu(\Upsilon_0,\lambda)>0$
(cf. Remark \ref{rem:fishcakes}).  Since $H$ normalizes $|A,U|$,
 clearly $H$ normalizes $|A,U|_S$, and it follows that if $\lambda\in |A,U|$ and $h\in H$, then $\mu(\Upsilon_0,\lambda)>0$ if and only if $\mu(\Upsilon_0,h^{-1}\cdot \lambda)>0$ if and only if $\mu(h_*\Upsilon_0,\lambda)>0$, where the last equivalence comes from Lemma \ref{lem:propertiesofstates}(ii).  Part (ii) now follows from Lemma \ref{lem:unionofstates}(iv).
\end{proof}

\begin{rem}
\label{rem:getkempf}
Using the quasi-states $\Xi$ and $\Upsilon$ from Proposition \ref{prop:kempfstate}, we can now
recover many of the existing optimality results from the literature by applying Theorem \ref{thm:kempf2.2}.
For example, to get Kempf's \cite[Thm.\ 3.4]{kempf}, we consider the case that $U = \{x\}$ is a singleton:
then Theorem \ref{thm:kempf2.2} supplies us with an optimal class $\Lambda$ of cocharacters attached to $x$,
and the corresponding optimal parabolic subgroup $P$ of $G$ contains the stabilizer $C_G(x)$,
by Proposition \ref{prop:kempfstate}
and Remark \ref{rem:funct}.  If $\lambda\in \Lambda$, then $n\lambda\in |A,U|$ for some $n\in \NN$; Proposition \ref{prop:kempfstate}(ii) ensures that $n\lambda$ actually belongs to $|A,U|_S$, as is required in \cite[Thm.\ 3.4]{kempf}.  In the more general setting that $U$ is an arbitrary subset of $A$, we obtain results on uniform $S$-instability from \cite{GIT}.
In this case, again thanks to Theorem \ref{thm:kempf2.2} and Proposition \ref{prop:kempfstate}, we obtain
\cite[Thm.\ 4.5]{GIT}.

We have now also set up all the necessary preliminaries to fully interpret the results of Kempf \cite{kempf} and
Hesselink \cite{He} in the language of buildings.
Recall that we set $E_{A,U}(K) = \zeta(D_{A,U}(K)\setminus\{0\}) \subseteq \Delta(K)$.
Now Lemma \ref{lem:ThetaAU} says that $E_{A,U}(K)$ is a convex polyhedral subset of finite type in $\Delta(K)$,
and Theorem \ref{thm:partialTCC} combined with Proposition \ref{prop:kempfstate} says that $E_{A,U}(K)$ has a $G$-centre if
$U$ is properly uniformly $S$-unstable.
Note that in this case, we also have that $E_{A,U}(K)$ is not $\Delta(K)$-cr.
This follows using similar arguments to those in Remark \ref{rem:needtwostates}:
there is an admissible quasi-state whose numerical function attains a finite positive value on
the cone $D_{A,U}(K)$.
Thus, interpreted in the building $\Delta(K)$, Kempf's result \cite[Thm.\ 3.4]{kempf} really is proving a
special case of the Centre Conjecture \ref{conj:tcc}.
\end{rem}

\begin{rem}
Keeping the notation from the previous remark,
it is worth stressing here that $E_{A,U}(K)$ is \emph{not}
a subcomplex of $\Delta(K)$ in general, so the methods in this section apply to
cases of Conjecture \ref{conj:tcc} not covered by Theorem \ref{thm:knowntcc}.

For an easy example of this, let $G = \SL_3(k)$ acting on its natural module $V = k^3$, and let $v = (1,1,0) \in V$.
Consider the cocharacters $\lambda$ and $\mu \in Y$ given by
$\lambda(a) = \diag(a^2,a,a^{-3})$ and $\mu(a) = \diag(a^3,a^{-1},a^{-2})$ for $a \in k^*$.
Then $P_\lambda = P_\mu$ is the Borel subgroup of $G$ consisting of upper triangular matrices,
but $\lambda$ destabilizes $v$ whereas $\mu$ does not.
Hence $E_{V,v}(K)$ does not contain the
whole simplex corresponding to $P_\lambda$, and
hence cannot be a subcomplex of $\Delta(K)$.
\end{rem}


\begin{thm}
\label{thm:precise}
Let $\Sigma$ be a convex polyhedral subset of finite type of
$\Delta(K)$, and let $C = \zeta\inverse(\Sigma) \cup \{0\}$.
Let $A$ be an affine $G$-variety,
$S$ a non-empty closed $G$-stable subvariety of $A$,
and $U$ a subset of $A$.
If $\Upsilon$ attains a finite positive value on $C$, where $\Upsilon = \Upsilon_{A,U,S}$
is the quasi-state from Proposition \ref{prop:kempfstate}(ii),
then $\Sigma$ has a $C_G(\Upsilon)$-centre.
If, further, $\mu(\Upsilon,\cdot)$ takes finite positive values on the whole of some $N_G(\Sigma)$-orbit
in $C$, then $\Sigma$ has a $G$-centre.
\end{thm}

\begin{proof}
For the first assertion, just apply Corollary \ref{cor:admissiblecase}.
For the second, replace $\Upsilon$ with $\Upsilon' := \bigcup_{g \in N_G(\Sigma)} g_*\Upsilon$
and note that $\Upsilon'$ is admissible and $N_G(\Sigma) \subseteq C_G(\Upsilon')$,
by Lemma \ref{lem:averaging}.
Since $\mu(\Upsilon,\cdot)$ takes finite positive values on the whole of some $N_G(\Sigma)$-orbit in $C$, so
does $\mu(\Upsilon',\cdot)$, by Lemmas \ref{lem:propertiesofstates}(ii) and \ref{lem:unionofstates}(iv).
Now apply Corollary \ref{cor:admissiblecase} to $\Sigma$ and $\Upsilon'$.
\end{proof}

\begin{rem}
\label{rem:twosettings}
We have two different settings where Theorem \ref{thm:precise} is useful.
First, suppose we have a convex polyhedral subset $\Sigma$ of
$\Delta(K)$ such that $\Sigma$ is not $\Delta(K)$-cr.
Then we want to find a $G$-centre of $\Sigma$.
Roughly speaking, Theorem \ref{thm:precise} says that we can do
this by finding suitable $A$, $S$ and $U$ such that
some element of $\zeta^\inverse(\Sigma)$ properly destabilizes $U$ into $S$.
For an example of this, see Theorem \ref{thm:tcc} below and Theorem \ref{thm:apt}.

Second, suppose that we have suitable $A$, $S$ and  $U$, as above,
and we want to find a $G$-centre of $E_{A,U}(K)$ subject to the extra condition that
this centre also lies in some convex subset $\Sigma\subseteq \Delta(K)$. Theorem \ref{thm:precise} helps us to do this.
For example, suppose $H$ is a reductive subgroup of $G$.
Then $\Sigma = \zeta(Y_H(K)\setminus\{0\})$ is a convex subset of $\Delta(K)$,
and if Theorem \ref{thm:precise} applies, it provides a $G$-centre of $E_{A,U}(K)$ which ``comes from''
a cocharacter of $H$.
See \cite{BMRT:relative} for similar ideas.
\end{rem}


We finish this section by describing how to apply our results to another case
of the Centre Conjecture \ref{conj:tcc} which does not appear to have anything to do
with GIT (Theorem \ref{thm:tcc} below).
The idea is that finding a suitable $G$-action on an affine variety $A$ can help to
establish the existence of a centre.

Recall the material on $G$-complete reducibility introduced in
Section \ref{sec:gcr}.
Theorem \ref{thm:tcc} asserts the existence of
a $G$-centre of the convex non-$\Delta(K)$-cr subset $\Sigma$ of $\Delta(K)$,
provided $\Sigma$ is fixed pointwise by a suitable subgroup
of $G$.
We make this precise in our next definition.

\begin{defn}
\label{def:G-contractible}
Let $\Sigma$ be a convex polyhedral subset of $\Delta(K)$, and let $H$ be a subgroup of $G$.
We say that $H$ \emph{witnesses the fact that $\Sigma$ is not $\Delta(K)$-cr} if
$\Sigma \subseteq \Delta(K)^H$
and there is a $y \in \Sigma$
which has no opposite in $\Delta(K)^H$.
Note that, in this case, neither $\Sigma$ nor $\Delta(K)^H$
is $\Delta(K)$-cr,
so in particular, $H$ is not $G$-cr, \cite[\S3]{serre2}.
\end{defn}



\begin{thm}
\label{thm:tcc}

Let $\Sigma \subseteq \Delta(K)$ be a convex polyhedral subset of
finite type.
If there exists a subgroup of $G$ which witnesses the fact that $\Sigma$ is not $\Delta(K)$-cr, then $\Sigma$ has a $G$-centre.
\end{thm}

\begin{proof}
Let $H$ be a subgroup of $G$ such that
$\Sigma \subseteq \Delta(K)^H$ and let
$y \in \Sigma$ such that $y$ has no opposite in $\Delta(K)^H$.  Let $C= \zeta^{-1}(\Sigma)\cup\{0\}$.
We may replace $H$ with the subgroup $\bigcap_{\nu\in C} P_\nu$
without affecting the hypotheses of the theorem;
this replacement ensures that $N_G(\Sigma) \subseteq N_G(H)$.
Let $\lambda \in C$ such that $y = \zeta(\lambda)\in \Sigma$, 
and let $T$ be a maximal torus of $G$ such that $\lambda\in Y_T(K)$.
Let $P = P_\lambda$ and $L = L_\lambda$.
Note that since $y$ has no opposite in $\Sigma$, $-(u\cdot\lambda) \not\in C$ for any $u \in R_u(P_\lambda)$.
We want to apply Theorem \ref{thm:precise}, so we need to verify
the conditions there.

We have $H\subseteq P$.  Suppose $H$ is contained in a Levi subgroup $M$ of $P$.
Then $M$ is of the form $M= L_{u\cdot \lambda}$ for some $u\in R_u(P)$, so $H$ fixes $u\cdot \lambda$, so $H$ fixes $-(u\cdot \lambda)$.
But $y$ has no opposite in $\Delta(K)^H$, so we have a contradiction.
Thus $H$ is not contained in any Levi subgroup of $P$.
Let $H'$ denote the image of $H$ under the canonical projection $P \to L$.
Then $H$ and $H'$ are not conjugate, by \cite[Thm.\ 5.8]{GIT}.


Now pick $n\in \NN$ such that $H$ admits a generic $n$-tuple (see Section \ref{sec:gcr}),
and recall that $G$ acts on $G^n$ by simultaneous conjugation.
There exists $\nu \in Y_T$ (a genuine cocharacter)
such that $P_\nu = P$ and $L_\nu = L$.  Taking the limit along $\nu$ moves this generic tuple for $H$ into $(H')^n$.
Thus, if we set $S= \ovl{G\cdot (H')^n} \subseteq G^n$, we see that
$H^n$ is uniformly $S$-unstable.
Moreover, since $H$ and $H'$ are not $G$-conjugate,
it follows from the proof of \cite[Thm.\ 5.16]{GIT} that $\dim C_G(\mathbf{s}) > \dim C_G(H)$
for any $\mathbf{s} \in S$.
This means that $H^n$ is not contained in $S$, and
thus $H^n$ is \emph{properly} uniformly $S$-unstable.


By Proposition \ref{prop:kempfstate},
there is a bounded quasi-state
$\Upsilon$ such that
$|G^n,H^n|_S
= \{\nu \in |G^n,H^n|\mid \mu(\Upsilon,\nu)>0 \}$,
and $N_G(H) = N_G(H^n) \subseteq C_G(\Upsilon)$.
Since $\nu \in |G^n,H^n|_S$ for all $\nu \in Y_T$ with $P_\nu = P$,
and since we can scale any point of $Y_T(\QQ)$ by a positive integer to give
an element of $Y_T$, we can conclude that every
$\nu \in Y_T(\QQ)$ with $P_\nu = P$ satisfies $\mu(\Upsilon,\nu) >0$.  Hence $\mu(\Upsilon,\lambda)>0$ if $K= \QQ$.
If $K= \RR$, then we need a more complicated argument.
It follows from the proof of Lemma \ref{lem:cxcvx} that for some $r$, the cone $C^P_T$
is generated by cocharacters
$\tau_1,\ldots, \tau_r\in Y_T$ with the property that for any
$\nu\in Y_T(\RR)$, we have $P_\nu = P$ if and only if
$\nu = \sum a_i\tau_i$ with all the $a_i >0$.
In particular, we have $\lambda = \sum a_i\tau_i$ with $a_i >0$ for each $i$.
Let $\epsilon_1,\ldots, \epsilon_r$ be positive rational numbers and define
$\tau_1',\ldots, \tau_r'$ by
$\tau_i'= \tau_i+ \sum_{j\neq i} \epsilon_j \tau_j$.
If we choose the $\epsilon_i$ to be sufficiently small, then we have
$\lambda= \sum_{i=1}^r a_i' \tau_i'$
for some $a_1',\ldots, a_r'>0$.
Since $P_{\tau_i'}= P$ for all $i$, and each $\tau_i' \in Y_T(\QQ)$,
we have $\mu(\Upsilon,\tau_i')>0$ for all $i$ by the arguments above.
Repeated application of Lemma \ref{lem:propertiesofstates}(i) gives
$\mu(\Upsilon,\lambda)>0$.

Finally, note that $N_G(\Sigma) \subseteq N_G(H) \subseteq C_G(\Upsilon)$, so $\mu(\Upsilon,\cdot)$ takes positive values on the orbit $N_G(\Sigma)\cdot \lambda$.
We have now put all the conditions in place to apply Theorem \ref{thm:precise},
which finishes the proof.
\end{proof}

\begin{rem}
Theorem \ref{thm:tcc}
generalizes the main result from \cite{BMR:tits} and also \cite[Thm. 5.31]{GIT}.
In \cite[Thm.\ 3.1]{BMR:tits}, we proved the special case of
Theorem \ref{thm:tcc} when $\Sigma = \Delta(K)^H$.
Note that, under the hypotheses of Theorem \ref{thm:tcc}, $\Delta(K)^H$ is a convex non-$\Delta(K)$-cr subcomplex of $\Delta(K)$.
Thus by Theorem \ref{thm:knowntcc}, $\Delta(K)^H$ admits
an $\Aut \Delta(K)$-centre. However, it does not follow in general that
this centre lies in $\Sigma$,
so Theorem \ref{thm:knowntcc} cannot be applied to find a centre of $\Sigma$.
\end{rem}


\section{Extensions}
\label{sec:extra}

In this section, we briefly discuss various ways in which our work in this paper can be extended.
We will return to these ideas in future work.

\subsection{Reductive groups}
\label{sub:redgps}
For simplicity, we have restricted attention in this paper to the
case when the group $G$ is semisimple.
However, in the setting of GIT, one often considers a reductive group acting on
an affine variety such that the centre does not act trivially (just consider the
action of $\GL(V)$ on the natural module $V$).
Many of our results go through under the weaker assumption that $G$ is a reductive group.
In particular, Theorem \ref{thm:kempf2.2} works for a reductive group $G$, and so the
later results that rely on it also go through.

One reason for restricting attention to the case that $G$ is semisimple is
that this facilitates our construction of the building $\Delta(K)$ of $G$ from the
set of cocharacters $Y(K)$ in Section \ref{sub:buildings}.
If $G$ is reductive but not semisimple, then the object $\Delta(K)$ we construct actually
contains a contribution from the centre of $G$ (see \cite[Sec.\ IV, Remarques]{rousseau}).
Considering convex subsets of this new object suggests a generalization of the Centre Conjecture \ref{conj:tcc}.
Our results are easily seen to go through in this case (in particular, see Theorem \ref{thm:bldgversion},
Theorem \ref{thm:partialTCC}, and the material in Section \ref{sec:gitstcc} above).

\subsection{Automorphisms of $G$}
As remarked in the previous section, if one is primarily interested
in the building of $G$, it is no real loss to assume that
$G$ is semisimple.
Further, the isogeny class of $G$ does not change the structure of the building,
so we can also assume $G$ is adjoint.
This allows us to view $G$ as a subgroup of $\Aut(G)$, the (algebraic) group of all
algebraic automorphisms of $G$.
Many of our constructions extend to give $\Aut(G)$-centres rather than $G$-centres.
The crucial observations that allow us to make this transition are that
the actions of $G$ on $Y$ and $X$ extend naturally to actions of $\Aut(G)$, and that
we can take the norm $\left\|\,\right\|$ on $Y$ in Definition \ref{def:norm}
to be $\Aut(G)$-invariant; see \cite[Sect.~7]{rich1}.
The functoriality of our constructions under the action of $G$ noted in Remark \ref{rem:funct}
extends to $\Aut(G)$-functoriality.

These facts allow us to extend our results about $G$-centres to results about $\Aut(G)$-centres
without much effort.
In particular, under the assumption that $G$ is semisimple and adjoint, we can suitably modify
Theorems \ref{thm:partialTCC}, \ref{thm:precise}, \ref{thm:tcc},
and \ref{thm:apt}
so that they provide $\Aut(G)$-centres for the subsets $\Sigma$ involved.
This is another step towards the full version of Conjecture \ref{conj:tcc} for the buildings $\Delta(K)$ in this paper,
as $\Aut \Delta(K)$ is made up of $\Aut(G)$ together with field automorphisms, \cite[Sec.\ 5]{tits1}.
See also \cite{BMR:tits} and \cite{GIT} for constructions involving $\Aut(G)$.

\subsection{Field automorphisms}\label{sub:fieldauts}
It is clear from the existing literature that our optimality
constructions behave well with respect to the induced
action of Galois groups; see \cite{kempf}, \cite{rousseau}, \cite{He}, \cite{GIT}.
More precisely, let $k$ be a field and let $G$ be defined over $k$.
Let $\Gamma$ denote the Galois group $\Gal(k_s/k)$,
where $k_s$ denotes the separable closure of $k$ in its algebraic closure.
Then $\Gamma$ also acts on the set of cocharacters $Y$ of $G$
and we can ensure that the norm is invariant under this action, cf.\ \cite[Def.\ 4.1]{GIT}.
Following \cite[5.7.1]{tits1}, any $\gamma \in \Gamma$ induces an automorphism of the building $\Delta(K)$,
and the $\Gamma$-invariance of the norm ensures that, where this makes sense, the $G$-centres we find in this paper
are also $\Gamma$-invariant.
We can thus make further progress towards proving the existence of $\Gamma$-centres of convex subsets of $\Delta(K)$.

\subsection{Non-algebraically closed fields}
We now indicate briefly how our work carries over to the case of a non-algebraically closed field.
Let $k$ be an arbitrary field and let $\ovl{k}$ be the algebraic closure of $k$.
Let $G$ be a semisimple algebraic group defined over $k$.
If $H$ is a subgroup of $G$, then we denote by $Y_{H,k}$ the subset of $Y_H$ consisting of the $k$-defined cocharacters of $H$,
and by $X_{H,k}$ the subset of $X_H$ consisting of the $k$-defined characters of $H$.
We write simply $Y_k$ for $Y_{G,k}$ and $X_k$ for the union of the $X_{T,k}$ as $T$ runs over the ($k$-defined) maximal tori of $G$.
We have a spherical building $\Delta_k= \Delta_{G,k}$ associated to $G$:
the simplices correspond to the $k$-defined parabolic subgroups of $G$, ordered by reverse inclusion,
and the apartments correspond to the maximal $k$-split tori of $G$ \cite{tits1}.
It is clear from this combinatorial description that $\Delta_k$ is a subbuilding of $\Delta$.
This inclusion holds at the level of geometric realizations as well, as we now explain.
The construction of $Y(\QQ)$ and $Y(\RR)$ carries over without problems to give spaces $Y_k(\QQ)$ and $Y_k(\RR)$,
and we can construct vector buildings $V_k(\QQ)$ and $V_k(\RR)$ and spherical buildings $\Delta_k(\QQ)$ and $\Delta_k(\RR)$.
The obvious maps from $Y_k(K)$ to $Y(K)$, $V_k(K)$ to $V(K)$ and $\Delta_k(K)$ to $\Delta(K)$ are all inclusions.
To see this, note that if $P$ is a $k$-defined parabolic subgroup of $G$, then there exists $\lambda\in Y_k$ such that
$P= P_\lambda$ \cite[15.1.2(ii)]{spr2}, and there is a maximal $k$-split torus $T$ of $G$ such that $\lambda\in Y_k(T)$;
moreover, if $x\in P_\lambda$ (resp.\ $x\in R_u(P_\lambda)$) such that $u\cdot\lambda$ is $k$-defined, then there exists $y\in P_\lambda(k)$ (resp.\ $y\in R_u(P_\lambda)(k)$) such that $x\cdot \lambda= y\cdot \lambda$.

If $G$ is $k$-split, then any maximal $k$-split torus $T$ of $G$ is $k$-split.
In this case, the notions and properties of states and quasi-states carry over in a natural way,
since every character and cocharacter of $T$ is then automatically $k$-defined.
For problems involving centres, one can often reduce to this case.
One takes a Galois extension $k'/k$ and works in the spherical building $\Delta_{k'}$;
having found a natural centre $c\in \Delta_{k'}$, one then shows that $c$ is invariant under
${\rm Gal}(k'/k)$ and deduces that $c$ lies in the subbuilding $\Delta_k$ of $\Delta_{k'}$
(cf.\ Section \ref{sub:fieldauts} and the proof of Theorem \ref{thm:perfect}).

Many of our results carry over to the non-algebraically closed case with minor modifications.
Some subtleties occur, however, when considering the results from Section \ref{sec:gitstcc}.
These were studied in \cite{GIT}, but in the language of cocharacters and GIT.
Here we translate them into the language of buildings.
Let $G$ act on an affine variety $A$ over $k$ and let $U\subseteq A$.
We define subsets $D_{A,U,k}(K)$ of $Y_k(K)$ and $E_{A,U,k}(K)$ of $\Delta_k(K)$ by
$D_{A,U,k}(K)= D_{A,U}(K)\cap Y_k(K)$ and $E_{A,U,k}(K)= E_{A,U}(K)\cap \Delta_k(K)= \zeta(D_{A,U,k}(K)\setminus\{0\})$.

Hesselink formulated Kempf's results over an arbitrary field.  In particular, he proved the following theorem, which generalizes \cite[Thm.~3.4]{kempf} (cf.\ \cite[Thm.~4.5]{He}, \cite[Thm.~5.5]{GIT}).

\begin{thm}
\label{thm:kempforig}
 Let $A$ be an affine $G$-variety over $k$ and let $x\in A(k)$.
 Suppose there exists $\lambda\in Y_k$ such that $\lambda$ properly destabilizes $x$.
 Then $E_{A,x,k}(\QQ)$ has a $G(k)$-centre.
\end{thm}

The proof follows immediately from Theorem \ref{thm:kempf2.2} and Proposition \ref{prop:kempfstate} in the algebraically closed case:
we simply note that we may regard $E_{A,x,k}(\QQ)$ as a convex subset of $\Delta(\QQ)= \Delta_{\ovl{k}}(\QQ)$ and that a $G$-centre of
$E_{A,x,k}(\QQ)$ is also a $G(k)$-centre of $E_{A,x,k}(\QQ)$.

This is not, however, an exact counterpart of \cite[Thm.~3.4]{kempf} for an arbitrary $k$.
We see this most easily using the language of buildings.
Let $A,x$ be as in Theorem \ref{thm:kempforig} and let $\lambda\in Y_k$ such that $\lambda$ properly destabilizes $x$.
It follows easily using \cite[Lem.~2.12]{GIT} that $\zeta(\lambda)$ has no opposite in $E_{A,x}(\QQ)$.
This implies that $E_{A,x}(\QQ)$ is not $\Delta(\QQ)$-cr and that $E_{A,x,k}(\QQ)$ is not $\Delta_k(\QQ)$-cr.
There exist examples of $A$ and $x$ where $E_{A,x,k}(\QQ)$ is not $\Delta_k(\QQ)$-cr but $E_{A,x}(\QQ)$ is $\Delta(\QQ)$-cr
(cf.\ \cite[Rem.~5.10]{GIT}).
It is natural to ask whether the weaker hypothesis that $E_{A,x,k}(\QQ)$ is not $\Delta_k(\QQ)$-cr alone implies that $E_{A,x,k}(\QQ)$ has a $G(k)$-centre; this would be the most natural extension of \cite[Thm.~3.4]{kempf} to an arbitrary field.
We do not know the answer in general, but we finish by giving the answer in the case of a perfect field.

\begin{thm}
\label{thm:perfect}
 Suppose $k$ is perfect.
 Let $A$ be an affine $G$-variety over $k$ and let $x\in A(k)$.
  If $E_{A,x,k}(\QQ)$ is not $\Delta_k(\QQ)$-completely reducible, then $E_{A,x,k}(\QQ)$ has a $G(k)$-centre.
\end{thm}

\begin{proof}
 Suppose $E_{A,x,k}(\QQ)$ is not $\Delta_k(\QQ)$-completely reducible.
 Then there exists $\lambda\in Y_k$ such that $\lambda$ destabilizes $x$ and $\zeta(\lambda)$ has no opposite in $E_{A,x,k}(\QQ)$.
 Set $x':= \lim_{a\ra 0} \lambda(a)\cdot x$.
 Let $u\in R_u(P_\lambda)(k)$.
 Then $-(u\cdot\lambda)$ does not destabilize $x$, so $u\cdot \lambda$ does not fix $x$,
 so $x'\neq u^{-1}\cdot x$, by \cite[Lem.~2.12]{GIT}.
 Hence $x'$ is not $R_u(P_\lambda)(k)$-conjugate to $x$.
 Since $k$ is perfect, \cite[Thm.~3.1]{GIT} implies that $x'$ is not $R_u(P_\lambda)$-conjugate to $x$.
 Hence by \cite[Thm.~3.3]{GIT}, $x'$ is not $G$-conjugate to $x$.
 This means that $\lambda$ properly destabilizes $x$.
 The result now follows from Theorem \ref{thm:kempforig}.
\end{proof}

\bigskip
{\bf Acknowledgements}:
The authors acknowledge the financial support of
Marsden Grant UOC0501, The Royal Society and
the DFG-priority program SPP1388 ``Representation Theory''.
Parts of this paper were written during a stay of the
authors at the Isaac Newton Institute for Mathematical
Sciences in Cambridge during the ``Algebraic Lie Theory'' Programme in 2009,
and during a stay of the first and third authors at the
Max Planck Institute for Mathematics in Bonn in 2010.
We are indebted to Rudolf Tange for helpful discussions.


\end{document}